\documentclass[11pt,letterpaper]{amsart}

\usepackage[utf8]{inputenc}
\usepackage[english]{babel}
\usepackage{amsmath}
\usepackage{amsthm}
\usepackage{amssymb}
\usepackage{mathscinet}
\usepackage{forest}
\usepackage{tikz}
\usetikzlibrary{positioning}

\usepackage{float}
\usepackage[hidelinks, pagebackref]{hyperref}

\renewcommand*{\backref}[1]{}
\renewcommand*{\backrefalt}[4]{
  \ifcase #1 
  [No citations.]
  \or [#2]
  \else [#2]
  \fi }

\usepackage{microtype}
\usepackage[inline]{enumitem}

\theoremstyle{plain}
\newtheorem{theorem}{Theorem}[section]
\newtheorem{proposition}[theorem]{Proposition}
\newtheorem{lemma}[theorem]{Lemma}
\newtheorem{corollary}[theorem]{Corollary}

\theoremstyle{definition}
\newtheorem{definition}[theorem]{Definition}
\newtheorem{remark}[theorem]{Remark}
\newtheorem*{remark*}{Remark}
\newtheorem{example}[theorem]{Example}
\newtheorem*{example*}{Example}

\makeatletter
\let\c@equation\c@subsection
\makeatother
\numberwithin{theorem}{section} 

\makeatletter
\let\c@figure\c@equation
\makeatother
\numberwithin{figure}{section} 

\makeatletter
\newtheoremstyle{case-style}
  {3pt}
  {3pt}
  {}
  {}
  {\itshape}
  {.}
  {.5em}
  {\thmname{#1}\thmnumber{\@ifnotempty{#1}{ }#2}%
   \thmnote{ {\the\thm@notefont(#3)}}}
\makeatother

\theoremstyle{case-style}
\newtheorem{case}{Case}
\newtheorem{subcase}{Case}
\newtheorem{case_t}{Case}
\newtheorem{subcase_t}{Case}
\newtheorem{case_s}{Case}

\newtheorem{step}{Step}

\numberwithin{subcase}{case}
\numberwithin{subcase_t}{case_t}
\numberwithin{subcase_s}{case_s}

\theoremstyle{plain}
\newtheorem{XXXtheoremQED}[equation]{Theorem} 
  {\pushQED{\qed}\begin{XXXtheoremQED}}
  {\popQED\end{XXXtheoremQED}}

\newcommand{\fakeenv}{} 

\newenvironment{restate}[2]  
{ 
 \renewcommand{\fakeenv}{#2} 
 \theoremstyle{plain} 
 \newtheorem*{\fakeenv}{#1~\ref{#2}} 
 \begin{\fakeenv}
}
{
 \end{\fakeenv}
}

\newcommand{\Aut}{\mathrm{Aut}} 
\newcommand{\Out}{\mathrm{Out}} 
\newcommand{\SL}{\mathrm{SL}} 

\newcommand{\EXPSPACE}{\textsc{EXPSPACE}} 
\newcommand{\LogCFL}{\textsc{LogCFL}} 
\newcommand{\DET}{\textsc{DET}} 
\newcommand{\NP}{\textsc{NP}} 
\newcommand{\NC}{\textsc{NC}} 
\newcommand{\coNP}{\textsc{coNP}} 
\newcommand{\coRP}{\textsc{coRP}} 
\newcommand{\Ptime}{\textsc{P}} 
\newcommand{\PSPACE}{\textsc{PSPACE}} 

\newcommand{\depth}{\mathsf{depth}}
\newcommand{\height}{\mathsf{height}}
\newcommand{\eval}{\mathsf{eval}}

\makeatletter
\def\thmhead@plain#1#2#3{%
  \thmname{#1}\thmnumber{\@ifnotempty{#1}{ }\@upn{#2}}%
  \thmnote{ {\the\thm@notefont#3}}}
\let\thmhead\thmhead@plain
\makeatother

\newcommand{\from}{\colon}
\newcommand{\st}{\mathbin{\mid}} 

\newcommand{\calG}{{\mathcal{G}}}
\newcommand{\calH}{{\mathcal{H}}}
\newcommand{\calL}{{\mathcal{L}}}
\newcommand{\calM}{{\mathcal{M}}}

\newcommand{\NN}{\mathbb{N}}
\newcommand{\ZZ}{\mathbb{Z}}

\newcommand{\emptyword}{\varepsilon}
\newcommand{\slex}{\mathsf{slex}}
\newcommand{\Ball}{{\sf{B}}}

\newcommand{\teth}[1]{\langle #1 \rangle}
\newcommand{\subgp}[1]{\langle #1 \rangle}
\newcommand{\group}[2]{\left\langle #1 \left\vert\, #2 \right. \right\rangle}
\newcommand{\ceiling}[1]{\lceil #1 \rceil}

\usetikzlibrary{decorations.markings}
\makeatletter
\tikzset{%
    measureme/.style={
    decoration={
        markings,
        mark=at position 1 with {\node[below,black,anchor=west] {\quad \pgfdecoratedpathlength};},
    },
    postaction=decorate
}
}
\makeatother

\begin{document}

\title{Compressed decision problems in hyperbolic groups}

\author[D. Holt]{Derek Holt}
\address{University of Warwick, UK}
\email{D.F.Holt@warwick.ac.uk}
\author[M. Lohrey]{Markus Lohrey}
\address{Universit\"at Siegen, Germany}
\email{lohrey@eti.uni-siegen.de}
\author[S. Schleimer]{Saul Schleimer}
\address{University of Warwick, UK}
\email{s.schleimer@warwick.ac.uk}
\thanks{The second author has been supported by the DFG research project
LO 748/13-1.}

\begin{abstract}
We prove that, for any hyperbolic group, the compressed word and the compressed conjugacy problems are solvable in polynomial time.
As a consequence the word problem for the (outer) automorphism group of a hyperbolic group is solvable in polynomial time.
We also prove that the compressed simultaneous conjugacy and the compressed centraliser problems are solvable in polynomial time. 
Finally, we prove that, for any infinite hyperbolic group, the compressed knapsack problem is $\NP$--complete.


\end{abstract}

\maketitle

\section{Introduction}

\subsection{Background}

Suppose that $G$ is a finitely generated group.  
Let $\Sigma$ be a finite generating set which is \emph{symmetric}:
if $a$ lies in $\Sigma$ then so does $a^{-1}$. 
The \emph{word problem} for $G$ asks, given a word $w \in \Sigma^*$, if $w$ represents the identity in $G$. 
This, along with the \emph{conjugacy} and \emph{isomorphism problems}, was set out by Dehn~\cite{Dehn11} in 1911. 
These three decision problems are fundamental in group theory generally~\cite{MKS76, ChandlerMagnus82}. 

However, Dehn's claimed justification was that solutions to these problems have applications in what is now called \emph{low-dimensional topology}.  
Dehn's techniques, in particular his solution to the word problem in surface groups, were greatly generalised by Gromov.  
In~\cite{Gro87} Gromov introduced what are now called \emph{word-hyperbolic} or \emph{Gromov hyperbolic} groups.  
(We will simply call these \emph{hyperbolic} groups.  
See Section~\ref{sec-hyp}.)  
With this and other innovations Gromov revived the strictly geometric study of groups.  
For example, he characterised hyperbolic groups as being exactly those that satisfy a linear isoperimetric inequality; 
see~\cite[Section~6.8.M]{Gro87} as well as~\cite{Olshanskii91, Bowditch95, Papasoglu95}.
Gromov also showed that, in certain models of \emph{random groups}, all groups are almost surely hyperbolic; see~\cite[Section~5.5.F]{Gro87} as well as~\cite{Olsh92}.

Another theme in geometric group theory is the subject of \emph{distortion}.
As a concrete example, consider the Baumslag-Solitar group~\cite{BaumslagSolitar62}
\[
G = \group{a, b}{ b^{-1} a b = a^2 }
\]
The subgroup $\subgp{a}$ is exponentially distorted in $G$, in the sense that the element $a^{2^n} =_G b^{-n} a b^n$ has length $2^n$ as an element of $\subgp{a}$ but length $2n + 1$ as an element of $G$.
Thus, to solve the word problem efficiently in $G$ it seems necessary to record exponents of $a$, say in binary; see also \cite{Weiss16}.

So, seeking to solve the word problem in groups leads us to consider \emph{compressed words}: elements of the group given by some useful succinct representation. 
One popular such representation is by \emph{straight-line programs}; 
we give definitions and examples of these in Section~\ref{sec-SLP-def}.
We will call the word problem for group elements that are represented by straight-line programs the \emph{compressed word problem}.

The motivating result for this paper is a theorem due to the second author~\cite[Theorem~4.5]{Loh06siam}.  
Let $F_n$ be the free group on $n$ generators. 
Lohrey gives a polynomial-time algorithm to solve the compressed word problem for $F_n$; in fact this problem is $\Ptime$--complete.  
Building on this, the third author shows in~\cite[Theorems~5.2 and~6.1]{Schl06} that the word problems for $\Aut(F_n)$ and $\Out(F_n)$ can be solved
in polynomial time and then goes on to show that the compressed word problem for closed surface groups can be solved in polynomial time.
This also gives a new solution to the word problem in mapping class groups of surfaces.

This sequence of results closely parallels Dehn's original development, but in the compressed setting. 

\subsection{This paper}

Suppose that $w$ is a word in the generators of $G$.  
We say $w$ is \emph{shortlex reduced} if it is shorter than, or of the same length and lexicographically earlier than, any other word representing the same group element; see Definition~\ref{def-shortlex}.  
Suppose that $\calG$ is a straight-line program over $\Sigma$.  
Then we denote the output of $\calG$ by $\eval(\calG)$.  
Here is our main result.

\begin{restate}{Theorem}{thm-SLP-for-shortlex}
Let $G$ be a hyperbolic group, with symmetric generating set $\Sigma$. 
There is a polynomial-time algorithm that, given a straight-line program $\calG$ over $\Sigma$, finds a straight-line program $\calH$ so that $\eval(\calH)$ is the shortlex reduction of $\eval(\calG)$. 
\end{restate}

This was previously announced without proof in~\cite[Theorem~4.12]{Loh14}.  From this theorem we deduce the following.

\begin{restate}{Corollary}{coro-CWP-hyperbolic}
Let $G$ be a hyperbolic group. 
Then the compressed word problem for $G$ can be solved in polynomial time.
\end{restate} 

In the recent paper~\cite{HoltRees20} the first author, with Sarah Rees, has generalised the techniques of this paper to relatively hyperbolic groups where all peripheral groups are free abelian.  So, for a knot $K \subset S^3$, the compressed word problem for the knot complement is polynomial time.  This gives a further parallel with Dehn's programme for low-dimensional topology via the study of the fundamental group.

\subsection{Applications}

Given these results, in Section~\ref{sec-conjugacy} we deal with the compressed versions of several other algorithmic problems.  
Recall that the \emph{order problem} for a group $G$ asks us, given an element $g \in G$, to compute the order of $g$.  Since hyperbolic groups only have torsion elements of bounded order we can prove the following. 

\begin{restate}{Corollary}{coro-order}
Let $G$ be a hyperbolic group. 
Then the compressed order problem for $G$ can be solved in polynomial time.
\end{restate}

The first author, with Epstein, proved in~\cite{EpsteinH06} that the conjugacy problem in hyperbolic groups is linear time.
If, in the conjugacy problem, we replace the given pair of elements by a pair of finite ordered lists of elements then we obtain the \emph{simultaneous conjugacy problem}.  
See~\cite{KaMa12} and its references for a discussion of this problem, for various classes of groups. 

In the \emph{centraliser problem} the input consists of a list of group elements $g_1, \ldots, g_k \in G$ and the goal is
to compute a set of generators for the intersection of the centralizers of the $g_i$.
Holt and Buckley proved in~\cite{BuckleyHolt13} that the simultaneous conjugacy problem as well as the centraliser problem for hyperbolic groups is linear time.  

Using the results of~\cite{EpsteinH06, BuckleyHolt13}, and our work above, we solve the compressed versions of these problems. 

\begin{restate}{Theorem}{thm-compcp}
Let $G$ be a hyperbolic group. 
Then the compressed simultaneous conjugacy problem for $G$ can be solved in polynomial time.
Moreover, if the two input lists are conjugate, then we can compute a straight-line program for a conjugating element in polynomial time. 
\end{restate}

\begin{restate}{Theorem}{thm-comp-central}
Let $G$ be a hyperbolic group. 
Then the compressed centraliser problem for $G$ can be solved in polynomial time.
\end{restate}

We remark that, for finitely generated nilpotent groups, the (compressed) simultaneous conjugacy problem is solvable in polynomial time~\cite[Theorem~7]{MacDonaldMO17}.

As suggested in~\cite[Remark~A.5]{Schl06}, the word problem for a finitely generated subgroup of the automorphism group $\Aut(G)$ is polynomial-time reducible to the compressed word problem for $G$.
Similarly, the word problem for a finitely generated subgroup of the outer automorphism group $\Out(G)$ is polynomial-time reducible to the compressed simultaneous conjugacy problem for $G$; see~\cite[Proposition~10]{HauLohHau13}.  

Note that, if $G$ is hyperbolic then $\Aut(G)$, and thus $\Out(G)$, is finitely generated; see~\cite[Corollary~8.4]{DaGui11}.
We deduce the following. 

\begin{corollary}
Let $G$ be a hyperbolic group.  
Then the word problems for $\Aut(G)$ and $\Out(G)$ can be solved in polynomial time. \qed
\end{corollary}

Our final application is to knapsack problems. 
Suppose that $G$ is a finitely generated group.  
The given input is a list $(u_0, u_1, u_2, \ldots, u_k)$ of words over the generators of $G$.  
We are asked if there are natural numbers $n_i$ such that 
\[
u_0 =_G u_1^{n_1} u_2^{n_2} \cdots u_k^{n_k} 
\]
When $G$ is hyperbolic, the knapsack problem can be solved in polynomial time; see~\cite[Theorem~6.1]{MyNiUs14}. 

In the compressed knapsack problem, the words $u_i$ are represented by straight-line programs. 
For the special case $G = \ZZ$ this problem is a variant of the classical knapsack problem for binary encoded integers, which is $\NP$--complete~\cite[page~95]{Karp72}.
Using this, and our results above, we prove the following. 
 
\begin{restate}{Theorem}{thm-knapsack-hyp}
Let $G$ be an infinite hyperbolic group. 
Then the compressed knapsack problem for $G$ is $\NP$--complete.
\end{restate}

\subsection{Related work}

We here give a brief overview of previous work.  For a more in-depth treatment, we refer to~\cite{Loh12survey, Loh14}.

\subsubsection{Compressed word problems}
The use of straight-line programs in group theory dates back to, at least, the methods developed by Sims~\cite{Sims70} for computing with a subgroup of the symmetric group $S_n$ defined by generators. 
The first step in virtually all of the algorithms developed by Sims is to expand the given list of generators to a longer list (a \emph{strong generating set}) by defining a sequence of new generators as words in the existing generators.
Straight-line programs were later used, again in the context of finite groups, by Babai and Szemeredi~\cite{BabSzem84} in the proof of their \emph{Reachability Theorem}.

Note that the compressed word problem for a group $G$ is decidable if and only if the word problem for $G$ is decidable.
However, the computational complexity of the compressed word problem for $G$ can be strictly more difficult than the word problem itself.  We return to this topic below. 

It is interesting to note that the compressed word problem for a group $G$ is exactly the \emph{circuit evaluation problem} for $G$.
For finite groups the compressed word problem, and thus the circuit evaluation problem, is nearly linear time.
In fact, more is known. 
The parallel complexity of the circuit evaluation problem over finite groups is investigated in~\cite{BeMcPeTh97}. 
If $G$ is a finite solvable group, then the compressed word problem for $G$ belongs to the parallel complexity class $\DET \subseteq \NC^2$.  
If $G$ is finite and not solvable, then the compressed word problem for $G$ is $\Ptime$--complete. 

We now turn our attention to infinite, but finitely generated, groups.  
As mentioned above, the word problem for a finitely generated subgroup of $\Aut(G)$ is polynomial-time reducible to the compressed word problem for $G$.  
A similar reduction exists for certain group extensions~\cite[Theorem~4.8 and 4.9]{Loh14}.
These results on automorphisms are tightly connected to the study of distortion of subgroups, mentioned above. 

Beyond hyperbolic groups, there are several important classes of groups where the compressed word problem can be solved in polynomial time.  These include the following. 
\begin{itemize}
\item 
Finitely generated nilpotent groups~\cite[Section~4.7]{Loh14}.
Here, the compressed word problem belongs to the parallel complexity class $\DET$~\cite{KonigL15}.
\item 
Virtually special groups; that is, finite extensions of finitely generated subgroups of right-angled Artin groups~\cite[Corollary~5.6]{Loh14}.
Right-angled Artin groups are also known as graph groups or partially commutative groups. 
The class of virtually special groups contains all Coxeter groups~\cite{HagWi10}, one-relator groups with torsion~\cite{Wis09}, fully residually free groups~\cite{Wis09}, and fundamental groups of hyperbolic three-manifolds~\cite{Agol13}.
Note that the case of fully residually free groups is independently due to Macdonald~\cite{Macd09}.
\end{itemize}

Furthermore, the class of groups with polynomial time compressed word problem is closed under the following operations. 
\begin{itemize}
\item
Graph products~\cite[Chapter~5]{Loh14}.
\item
Amalgamated free products (or HNN--extensions) where the edges groups are finite~\cite[Chapter~6]{Loh14}.
\end{itemize}

We also note that, for finitely generated linear groups, the compressed word problem belongs to the complexity class $\coRP$~\cite[Theorem~4.15]{Loh14}.  
That is, there is a randomised polynomial-time algorithm that may err with a small probability on negative input instances.

\subsubsection{Hardness results}
Certain hardness results for the compressed word problem are known or suspected.
\begin{itemize}
\item 
The compressed word problem for every restricted wreath product $G \wr \ZZ$ with $G$ finitely generated nonabelian is $\coNP$--hard~\cite[Theorem~4.21]{Loh14}.  
For $G$ finite non-solvable (or free of rank two) the problem is $\PSPACE$--complete~\cite[Corollary~B]{BartholdiFLW19}; 
the authors obtain the same result for Thompson's group $F$, the Grigorchuk group and the Gupta-Sidki groups.

On the other hand: the uncompressed word problem for the Grigorchuk group can be solved in logarithmic space \cite{GaZa91}.  
Also, if $G$ is finite then the uncompressed word problem for $G \wr \ZZ$ belongs to the circuit complexity class $\NC^1$~\cite{Waa90}.  
Thus we have examples of groups where the compressed word problem is provably more difficult than the uncompressed word problem.

\item 
There exist automaton groups with an $\EXPSPACE$--complete compressed word problem~\cite{WaeWei19}. 

On the other hand, the uncompressed word problem for any automaton group belongs to $\PSPACE$.  
Again this gives examples where the compressed word problem is provably more difficult than the uncompressed. 

\item 
The compressed word problem for the linear group $\SL(3, \ZZ)$ is equivalent, up to polynomial-time reductions, to the problem of \emph{polynomial identity testing}.
This last is the decision problem of whether a given circuit over the polynomial ring $\ZZ[x_1, \ldots, x_n]$ evaluates to the zero polynomial~\cite[Theorem~4.16]{Loh14}. 
The existence of a polynomial-time algorithm for  polynomial identity testing is an outstanding open problem in the area of algebraic complexity theory.

On the other hand, the uncompressed word problem for $\SL(3, \ZZ)$ is polynomial time. 
\end{itemize}

\subsubsection{Knapsack problems over groups}

The uncompressed knapsack problem has been studied for various classes of groups; see~\cite{FrenkelNU15, GanardiKLZ18, KoenigLohreyZetzsche2015a}.
For non-elementary hyperbolic groups, the knapsack problem lies in $\LogCFL$ (the logspace closure of the class of context-free languages); see~\cite[Theorem~4.1]{Loh2019}. 
The second author further shows, in~\cite[Theorem~3.1]{LohreyZ18}, that the compressed knapsack problem for every virtually special group belongs to $\NP$. 

\subsubsection{Compressing integers}

In addition to straight-line programs, there are other methods of compression that arise in significant ways in computational group theory.  
Here we will mention just a few with particular relevance to the word problem.  
These are techniques for recording extremely large integers, as opposed to recording long words.  

The binary representation of an integer $n$ can be translated into a straight-line program $\calG_n$ of size $O(\log n)$ with output $a^n$.  
Following our discussion of circuit evaluation above, we could replace ``concatenation of strings'' by the primitive operator ``addition of integers''.  Likewise, we replace the alphabet $\{a\}$ by the alphabet $\{1\}$.  This transforms the straight-line program $\calG_n$ into an additive circuit with output $n$.  

If we allow multiplication as well as addition gates we obtain \emph{arithmetic circuits}.  
For example, a circuit with $n$ gates can produce an integer of size $2^{2^n}$ using iterated squaring.
\emph{Power circuits}, the topic of~\cite{MyUsWo}, replace multiplication of $x$ and $y$ by the operation $x \cdot 2^y$: 
that is, shifting the first input by the second.  
Thus a power circuit of depth $n$ can represent an integer of the size of a tower of exponentials of height $n$.
The same authors use their new theory, in~\cite{MyUsWo11}, to show that the word problem in Baumslag's group~\cite{Baumslag69}
\[
\group{a, b \,}{ \left( b^{-1} a b \right)^{-1} a \left( b^{-1} a b \right) = a^2 }
\]
is polynomial time. Recently the complexity has been further improved to the parallel complexity class $\textsc{TC}^1$ \cite{MattesW22}.
We note that Baumslag's group has a non-elementary Dehn function \cite{Plat04}; 
this demonstrates one of the many possible separations between the computational and the geometric theories of a group. 

Again exploiting various properties of power circuits, the authors of~\cite{DiekertLU12} give a cubic time algorithm for the word problem in Baumslag's group~\cite[Theorem~16]{DiekertLU12}.  They also show that the word problem for Higman's group~\cite{Higman51}
\[
\group{a, b, c, d}{ b^{-1}ab = a^2, c^{-1}bc = b^2, d^{-1}cd = c^2, a^{-1}da = d^2 }
\]
is polynomial time~\cite[Theorem~19]{DiekertLU12}.

Of course, even more extreme compression is possible and this leads to polynomial-time algorithms for even more extreme groups.  
The authors of \cite{DisonER16, DisonER18}  construct certain HNN--extensions of the \emph{hydra groups}~\cite{DisonRiley13}
\[
H_k = \group{a, b}{ [ \cdot \! \cdot \! \cdot [[a, b], b], \ldots , b] = 1 },
\]
where $[x, y] = x^{-1} y^{-1} x y$ is a \emph{commutator} and $[ \cdot \! \cdot \! \cdot [[a, b], b], \ldots , b]$
 is a nested commutator of depth $k$, for which the Dehn function grows roughly like the Ackermann function 
and the word problem is still solvable in polynomial time. For this, they use a compression scheme for integers that yields a compression ratio
of order of the Ackermann function on some integers.

\subsection{Acknowledgements}
An extended abstract of this paper appeared in \cite{HoltLS19}.
The third author thanks the first two for their patience during the writing of this paper. 

\section{General notation}
\label{sec-notation}

We include zero in the set of natural numbers; that is, $\NN = \{0, 1, 2, \ldots \}$.  

\subsection{Words}

Suppose that $\Sigma$ is an \emph{alphabet}; the elements of $\Sigma$ are called \emph{letters}.  
We write $\Sigma^*$ for the \emph{Kleene closure} of $\Sigma$; that is, the set of all finite words over $\Sigma$.  
We call any subset $L \subseteq \Sigma^*$ a \emph{language} over $\Sigma$. 

For any alphabet $\Sigma$, we use $\emptyword \in \Sigma^*$ to denote the \emph{empty word}.  
Suppose that $u$, $v$, and $w$ are words over $\Sigma$.   
We denote the \emph{concatenation} of $u$ and $v$ by $u \cdot v$; we often simplify this to just $uv$
($u \cdot v$ is sometimes preferred for better readability).
So, for example, $w \cdot \emptyword = \emptyword \cdot w = w$. 
We say that $u$ is a \emph{factor} of $v$ if there are words $x$ and $y$ so that $v = xuy$.
We say that $u$ is a \emph{rotation} of $v$ if there are words $x$ and $y$ so that $v = x y$ and $u = y x$.  
We have the following easy but useful result. 

\begin{lemma}
\label{lem-rotation}
Let $u$ and $v$ be words over $\Sigma$.  
Then $u$ is a rotation of $v$ if and only if $|u| = |v|$ and $u$ is a factor of $v \cdot v$.
\qed
\end{lemma}

Suppose that $w = a_0 \cdot a_1 \cdots a_{n-1}$ lies in $\Sigma^*$, where the $a_i$ are letters.  
Then we define $|w|$ to be the \emph{length} of $w$; that is, $|w| = n$.  
For any $i$ between zero and $n-1$ (inclusive) we define $w[i] = a_i$.  
Note that the empty word $\emptyword$ is the unique word of length zero. 

We now define the \emph{cut operators}.  
Let $w$ be a word, as above, and let $i$ and $j$ be indices with $0 \leq i \leq j \leq n = |w|$.  
We define $w[i:j] = a_i \cdots a_{j-1}$.  If $0 \leq i \leq j \leq |w|$ does not hold then $w[i:j]$ is not defined.
We use $w[:j]$ to denote $w[0:j]$, the \emph{prefix} of length $j$.  
We use $w[i:]$ to denote $w[i:n]$, the \emph{suffix} of length $n - i$.  
Note that $w[i:i] = \emptyword$ and $w = w[:i] \cdot w[i:]$.

Suppose that $\Sigma$ is a finite alphabet equipped with a total order $<$ (the concrete choice of $<$ will never be important for us).

\begin{definition}
\label{def-shortlex}
We define the \emph{shortlex order} on $\Sigma^*$ as follows.  For words $u, v$ we have $u <_\slex v$ if 
\begin{itemize}
\item 
$|u| < |v|$ or
\item 
$|u| = |v|$ and there are words $x, y, z \in \Sigma^*$ and letters $a, b \in \Sigma$ so that 
\begin{itemize}
\item
$u = xay$, 
\item
$v = xbz$, and 
\item
$a < b$.
\end{itemize}
\end{itemize}
\end{definition}

Note that shortlex is a \emph{well-order} on $\Sigma^*$: that is, every nonempty subset of $\Sigma^*$ has a unique shortlex least element. 

\subsection{Finite state automata}

We refer to~\cite{HoUl79} for background in automata theory. 
A \emph{(deterministic) finite state automaton} is a tuple $M = (Q, \Sigma, q_0, \delta, F)$, where 
\begin{itemize}
\item
$Q$ is a finite set of \emph{states}, 
\item
$\Sigma$ is a finite alphabet, 
\item
$q_0 \in Q$ is the \emph{initial} state, 
\item
$\delta \from Q \times \Sigma \to Q$ is a \emph{transition} function, and
\item
$F \subseteq Q$ is the set of \emph{accept} states. 
\end{itemize}

Intuitively, if the automaton $M$ is ``in'' state $q \in Q$, and receives input $a \in \Sigma$, then it transitions to the new state $\delta(q,a)$.
We extend $\delta$ to a function $\delta' \from Q \times \Sigma^* \to Q$ recursively.  That is, for any state $q$, word $w$, and letter $a$ we have
\begin{itemize}
\item
$\delta'(q, \emptyword) = q$ and
\item
$\delta'(q, wa) = \delta( \delta'(q, w), a)$.
\end{itemize}
Since $\delta$ and $\delta'$ agree on words of length at most one, we will suppress $\delta'$ in what follows and instead reuse $\delta$. 
We define 
\[
L(M) = \{ w \in \Sigma^* \st \delta(q_0, w) \in F \}
\]
to be the language \emph{accepted} by $M$.
Intuitively, if $w$ lies in $L(M)$ then $w$, when input into $M$, takes it from the initial state to an accept state. 

We say that a language $L \subseteq \Sigma^*$ is \emph{regular} if there exists a finite state automaton $M$ so that $L = L(M)$. 

\section{Hyperbolic groups} 
\label{sec-hyp}

We refer to~\cite{ABCFLMSS91} as a general reference on (word) hyperbolic groups.

Let $G$ be a finitely generated group.  
Let $1_G$ denote the identity element of $G$.  
Let $\Sigma$ be a finite, \emph{symmetric} generating set for $G$.  
That is, if $a$ lies in $\Sigma$ then so does $a^{-1}$.  
For two words $u, v \in \Sigma^*$, we will use $u =_G v$ to mean that $u$ and $v$ represent the same element of $G$.
We fix a total order $<$ on $\Sigma$.

The \emph{(right) Cayley graph} $\Gamma = \Gamma(G, \Sigma)$ of $G$ with respect to $\Sigma$ is defined as follows. 
\begin{itemize}
\item
The vertices of $\Gamma$ are the elements of $G$. 
\item
The undirected edges of $\Gamma$ are of the form $\{g, ga\}$ for $g \in G$ and $a \in \Sigma$.
\end{itemize}
We will label a directed edge $(g, ga)$ with the letter $a$.  
Note that $G$ acts, by graph automorphisms, on $\Gamma$ on the left.

Giving all edges length one makes $\Gamma$ into a geodesic metric space.  
We do this in such a way so that the action of $G$ is by isometries. 
The distance between two points $p, q$ is denoted $d_\Gamma(p, q)$.
For $g \in G$ we define $|g| = d_\Gamma(1,g)$.  
We deduce that $|g|$ is the smallest length among all words $w \in \Sigma^*$ that represent $g$. 
Fix $r \geq 0$.  
The ball of radius $r$ in $\Gamma$ is the set
\[
\Ball(r) = \Ball_\Gamma(r) = \{ g \in G : |g| \leq r \}
\]

Fix a word $w \in \Sigma^*$.  
We define $P_w \subseteq \Gamma$ to be the path starting at $1_G$ which is labelled by $w$. 
Thus the path $g \cdot P_w$ starts at $g$ and is again labelled by $w$.
In general, we will take $P \from [0, n] \to \Gamma$ to be an edge path from $P(0)$ to $P(n)$.  
In particular, we must allow real $t \in [0, n]$ as we traverse edges.  
We use $\bar{P}$ to denote $P$ with its parametrisation reversed.  
Note that $\bar{P}_w = g_w \cdot P_{w^{-1}}$.

We call a path $P$ \emph{geodesic} if for all real $t \ge 0$ we have $d_\Gamma(P(0), P(t)) = t$.  
Suppose that the word $w \in \Sigma^*$ represents the group element $g_w \in G$. 
We say that $w \in \Sigma^*$ is \emph{geodesic} if the path $P_w$ is geodesic.
We say that $w \in \Sigma^*$ is \emph{shortlex reduced} if for all $u \in \Sigma^*$ the equality $g_u = g_w$ implies $w \leq_\slex u$. 
We use $\slex(w)$ to denote the shortlex reduced representative of $g_w$. 

\begin{remark*}
Suppose that $w$ is geodesic or shortlex reduced.  
Suppose that $u$ is a factor of $w$.  
Then $u$ is also, respectively, geodesic or shortlex reduced.
\end{remark*}

A \emph{geodesic triangle} in $\Gamma$ consists of three \emph{vertices} $p, q, r \in G$ and three \emph{sides} $P, Q, R \subset \Gamma$. 
The sides are geodesic paths connecting the vertices; see Figure~\ref{fig-triangle}.
Fix $\delta \geq 0$.  
We now follow~\cite[Definition~1.3]{ABCFLMSS91}.  
We say that a geodesic triangle is \emph{$\delta$--slim}, if every point $x$ in the side $P$ is distance at most $\delta$ from some point of $R \cup Q$, and similarly for the sides $Q$ and $R$. 

Fix $G$ and $\Sigma$ as above.  
We say $G$ is \emph{$\delta$--hyperbolic} if  every geodesic triangle in the Cayley graph $\Gamma = \Gamma(G, \Sigma)$ is $\delta$--slim. 
Finally, we simply say $G$ is \emph{hyperbolic}, if it is  $\delta$--hyperbolic for some $\delta \geq 0$. 
For example, the group $G$ is $0$--hyperbolic (with respect to $\Sigma$) if and only if $G$ is a free group, freely generated by (half of) $\Sigma$. 

\begin{remark}
\label{rem-gromov}
In his article~\cite{Gro87}, Gromov proves that hyperbolic groups have many good properties.  In Corollary~2.3.B he states that such groups satisfy a linear isoperimetric inequality; hence they have solvable word problem.  In Corollary~2.3.E he shows that the notion of hyperbolicity is independent of the choice of finite generating set.  In Section~7.4.B he proves that they have solvable conjugacy problem.  For another exposition of these results (excepting the conjugacy problem) we refer to~\cite{ABCFLMSS91}: see Theorem~2.5, Proposition~2.10, and Theorem~2.18 of that work.  For an exposition of the conjugacy problem we refer to~\cite{EpsteinH06}.
\end{remark}

We will need a seemingly stronger condition on our geodesic triangles, called \emph{$\delta$--thinness}.  
We here follow~\cite[Definition~1.5]{ABCFLMSS91}.
Suppose again that we have a geodesic triangle with vertices $p, q, r \in G$ and with sides $P, Q, R \subset \Gamma$; see Figure~\ref{fig-triangle}.  
Let $c_P \in P$, $c_Q \in Q$, and $c_R \in R$ be the unique points so that
\[ 
d_\Gamma(p, c_Q) = d_\Gamma(p, c_R), \quad d_\Gamma(q, c_R) = d_\Gamma(q, c_P), \quad d_\Gamma(r, c_P) = d_\Gamma(r, c_Q)
\]
We call these the \emph{meeting points} of the triangle.  
Note that the meeting points may be elements of $G$ or midpoints of edges of $\Gamma$. 
Suppose that $x \in P$ and $y \in Q$ are points with 
\begin{itemize}
\item
$d_\Gamma(r, x) = d_\Gamma(r, y) = t$ and 
\item
$t \leq d_\Gamma(r, c_P) = d_\Gamma(r, c_Q)$. 
\end{itemize}
Then we call $x$ and $y$ \emph{corresponding points} with respect to $r$.  
Note that if one of $x$ or $y$ lies in $G$ then so does the other. 
We make the same definition with respect to the vertices $p$ and $q$. 
Note that the three meeting points are all in correspondence.
Fix $\delta \geq 0$.  
The triangle is called \emph{$\delta$--thin} if for all corresponding pairs $(x, y)$ we have $d_\Gamma(x,y) \leq \delta$.
See Figure~\ref{fig-triangle}; there the dotted arcs indicate corresponding pairs. 
Note that a $\delta$--thin triangle is $\delta$--slim.  
A converse also holds: every geodesic triangle in a $\delta$--hyperbolic space is $4\delta$--thin; see~\cite[Proposition 2.1]{ABCFLMSS91}. 

\begin{figure}[t]
  \centering{
    \scalebox{1}{
      \begin{tikzpicture}

        \tikzstyle{small} = [circle, draw = black, fill = black,inner sep=.25mm]
        \tikzstyle{smaller} = [circle, draw = black,fill=black,inner sep=.05mm]
        \tikzstyle{tiny} = [circle, draw = black,inner sep=.0mm]
        \tikzstyle{zero} = [circle, inner sep=0mm]

        \def\pq{185.0166}
        \def\pr{152.72462}
        \def\qr{164.06123}
        
        \def\p{86.84}
        \def\q{98.1766}
        \def\r{65.8846}

        \node (p) at (0,0) [small, label=left:$p$] {};
        \node (q) at (6,0) [small, label=right:$q$] {};
        \node (r) at (2,4) [small, label=above:$r$] {};
        \node (cp) at (3.11,1.92) [smaller, label=right:$c_P$] {};
        \node (cq) at (2.5,1.77) [smaller, label=left:$c_Q$] {};
        \node (cr) at (2.8,1.275) [smaller, label=below:$c_R$] {};

        \draw[decoration={markings, mark=at position \p-15 pt with \node[zero] (pq-1) {};, mark=at position \p-10 pt with \node[zero] (pq0) {};, mark=at position \p-5 pt with \node[zero] (pq1) {};, mark=at position \p pt with \node[tiny] (pq2) {};, mark=at position \pq-\q pt with \node[zero] (qp2) {};, mark=at position \p+5 pt with \node[zero] (qp1) {};, mark=at position \p+10 pt with \node[zero] (qp0) {};, mark=at position \p+15 pt with \node[zero] (qp-1) {};}, postaction={decorate}] 
(p) .. controls (3,1.7)  .. node[pos=.5,below=5mm] {$R$} (q);

        \draw[decoration={markings, mark=at position \p-15 pt with \node[zero] (pr-1) {};, mark=at position \p-10 pt with \node[zero] (pr0) {};, mark=at position \p-5 pt with \node[zero] (pr1) {};, mark=at position \p pt with \node[tiny] (pr2) {};, mark=at position \pr-\r pt with \node[zero] (rp2) {};, mark=at position \p+5 pt with \node[zero] (rp1) {};, mark=at position \p+10 pt with \node[zero] (rp0) {}; , mark=at position \p+15 pt with \node[zero] (rp-1) {};}, postaction={decorate}] 
(p) .. controls (3,1.7)  .. node[pos=.5,left=8mm] {$Q$} (r);

        \draw[decoration={markings, mark=at position \q-15 pt with \node[zero] (qr-1) {};, mark=at position \q-10 pt with \node[zero] (qr0) {};, mark=at position \q-5 pt with \node[zero] (qr1) {};, mark=at position \q pt with \node[tiny] (qr2) {};,mark=at position \qr-\r pt with \node[zero] (rq2) {};, mark=at position \q+5 pt with \node[zero] (rq1) {};, mark=at position \q+10 pt with \node[zero] (rq0) {};, mark=at position \q+15 pt with \node[zero] (rq-1) {};}, postaction={decorate}] 
(q) .. controls (3,1.7)  .. node[pos=.55,right=8mm] {$P$} (r);

        \draw[densely dotted, bend angle=20, bend right] (pq-1) to  (pr-1);
        \draw[densely dotted, bend angle=20, bend left] (qp-1) to  (qr-1);
        \draw[densely dotted, bend angle=20, bend right] (rp-1) to  (rq-1);
        
        \draw[densely dotted, bend angle=20, bend right] (pq0) to  (pr0);
        \draw[densely dotted, bend angle=20, bend left] (qp0) to  (qr0);
        \draw[densely dotted, bend angle=20, bend right] (rp0) to  (rq0);

        \draw[densely dotted, bend angle=20, bend right] (pq1) to  (pr1);
        \draw[densely dotted, bend angle=20, bend left] (qp1) to  (qr1);
        \draw[densely dotted, bend angle=20, bend right] (rp1) to  (rq1);

        \draw[densely dotted, bend angle=20, bend right] (pq2) to  (pr2);
        \draw[densely dotted, bend angle=20, bend left] (qp2) to  (qr2);
        \draw[densely dotted, bend angle=20, bend right] (rp2) to  (rq2);

      \end{tikzpicture}}}
  \caption{ A geodesic triangle in a hyperbolic metric space.  Note
    how the three sides ``bow in'' to a common centre.  Dotted lines
    represent paths of length at most $\delta$ between corresponding
    points. }
  \label{fig-triangle}
\end{figure}
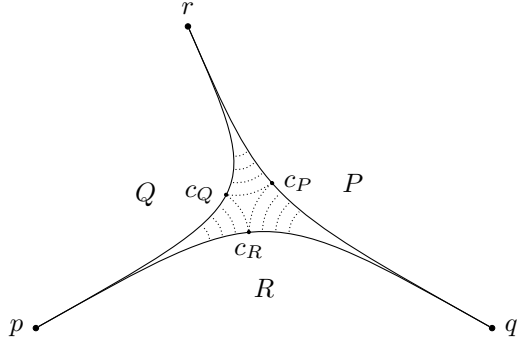

We now fix a group $G$ and a symmetric generating set $\Sigma$; we assume that $G$ is $\delta$--hyperbolic.  
We choose $\delta$ large enough to ensure that all geodesic triangles in $\Gamma$ are $\delta$--thin.  

\begin{remark*}
From a computational viewpoint, hyperbolic groups have many nice properties.  For example, their word problems can be solved in linear time~\cite[Theorem~2.18]{ABCFLMSS91} as can their conjugacy problems~\cite{EpsteinH06}.  (Here we gloss over the details of the required model of computation.)  In a more recent and noteworthy achievement, their isomorphism problem has also been solved; see~\cite{DaGui10, DaGui11}.  Thus all three of Dehn's fundamental problems have been settled positively for hyperbolic groups.  

Other positive results include the \emph{simultaneous conjugacy problem}~\cite{BrHo05,BuckleyHolt13} and the \emph{knapsack problem}~\cite{Loh2019}.  We will return to both of these below. 

Note that the compressed word problem easily reduces to the problem of checking the solvability of a system of equations. 
There is a substantial body of work on the latter, over hyperbolic groups.
Dahmani and Guirardel~\cite{DaGui10} prove (building on earlier work of~\cite{RiSe95}) that the problem is decidable.
The compressibility by straight-line programs of solutions of equations in hyperbolic groups is studied in~\cite{DiekertKM13}.
Ciobanu and Elder~\cite{CiobanuElder19} give a complete description of the set of all solutions of a given system of 
equations over a hyperbolic group.  
They obtain, as a corollary, a polynomial-space algorithm for deciding the existential theory of a hyperbolic group. 
\end{remark*}

The following results come from the fact that hyperbolic groups have automatic structures with respect to any shortlex ordering~\cite[Theorem~3.4.5 and Corollary~2.5.2]{ECHLPT}. 

\begin{lemma}[\mbox{\cite[Theorem~2.3.10]{ECHLPT}}] 
\label{lem-shortlex-reduction}
There is a polynomial-time (in fact, quadratic) algorithm that, given a word $w \in \Sigma^*$, produces $\slex(w)$.  \qed
\end{lemma}

\begin{lemma}[\mbox{\cite[Proposition~2.5.11 and Theorem~3.4.5]{ECHLPT}}]
\label{lem-shortlex-regular}
The languages in $\Sigma^*$, of geodesic words and of shortlex reduced words, are regular.  \qed
\end{lemma}

\begin{remark*}
We will in fact need both geodesic and shortlex reduced words in our proof of Theorem~\ref{thm-SLP-for-shortlex}.  
This is because the inverse of a geodesic word is again geodesic; this need not be the case for shortlex reduced words.  
On the other hand, shortlex reduced words provide unique representatives of group elements; 
this is almost never the case for geodesic words.
\end{remark*}

\begin{figure}
  \centering{
    \scalebox{1}{
      \begin{tikzpicture}
        \tikzstyle{small} = [circle,draw=black,fill=black,inner sep=.15mm]
        \tikzstyle{zero} = [circle,inner sep=0mm]

        \node[small] (1) {} ;
        \node[small,  right = 3cm of 1] (a) {} ;
        \node[small,  above = 0.68cm of a] (b) {} ;
        \node[small,  right = 6cm of 1] (2) {};
        \node[small,  above = 1.5cm of 1] (3) {};
        \node[small,  right = 6cm of 3] (4) {};
    
        \node[zero,  below = .6mm of 1] (1') {};
        \node[zero,  below = .6mm of 2] (2') {};
        \node[zero,  above = .6mm of 3] (3') {};
        \node[zero,  above = .6mm of 4] (4') {};
      
        \draw [->] (3) to [out=-27, in=-180] node[pos = 0.5, below = -0.7mm] {$v'$} (b);
        \draw [->] (b) to [out=0, in=-153] node[pos = 0.5, below = -0.7mm]  {$v''$} (4);
        \draw [->] (3') to [out=-27, in=-153] node[pos = 0.33, above = -0.7mm]{$v$} (4'); 
      
        \draw [->] (3) edge node[left=-.7mm]{$a$} (1);
        \draw [<-] (2) edge node[right=-.7mm]{$b$} (4);
        \draw [->] (1') edge node[pos = 0.33, below=-.7mm]{$u$} (2');
        \draw [->] (1) edge node[above=-.7mm]{$u'$} (a);
        \draw [->] (a) edge node[above=-.7mm]{$u''$} (2);
        \draw [<-] (a) edge node[left=-.7mm]{$c$} (b);

  \end{tikzpicture}}}
  \caption{Splitting a geodesic quadrilateral according to Lemma~\ref{lem-quad}.}
  \label{fig-quad}
\end{figure}
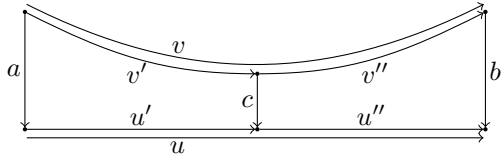

We will need the following standard lemma on geodesic quadrilaterals.  See, for example, the proof of~\cite[Proposition~3.5]{ABCFLMSS91}.

\begin{lemma} 
\label{lem-quad}
Let $a, b, u, v \in \Sigma^*$ be geodesic words such that $v b =_G a u$.  
Suppose that $u$ has a factorisation $u = u' u''$ with $|u'| \geq |a| + 2 \delta$ and $|u''| \geq |b| + 2 \delta$.
Then there exists a factorisation $v = v' v''$ and a geodesic word $c$ so that
\begin{itemize}
\item
$|c| \leq 2\delta$,
\item
$v' c =_G a u'$, and
\item
$v'' b =_G c u''$. 
\end{itemize}
\end{lemma}

\begin{proof}
We consider the quadrilateral with sides $P_a$, $g_a \cdot P_u$, $g_{v} \cdot P_b$, and
$P_v$.  
Here $g_w$ is the group element represented by a word $w$.  
We are given a factorisation $u = u' u''$.  Set $g = g_a g_{u'}$ and note that $g$ lies in $g_a \cdot P_u$.  
See Figure~\ref{fig-quad}, where we label a path by the word labelling the path (for instance, $g_a \cdot P_u$ is labelled with $u$).
Since geodesic quadrilaterals are $2\delta$--slim there is a group element $h$ with $d_\Gamma(g, h) \leq 2\delta$ lying in the union of the three other sides.

We now consider cases. 
Suppose that $h$ lies in $P_a - \{1_G\}$.  
Then the triangle inequality implies $|u'| < |a| + 2\delta$.
Similarly, if $h$ lies in $g_{v} \cdot P_b - \{g_{v}\}$, then $|u''| < |b| + 2\delta$.  
Both of these are contrary to hypothesis. 
We deduce that $h$ lies in $P_v$, proving the lemma. 
\end{proof}

The lemma has a useful corollary.

\begin{corollary} 
\label{cor-quad-two}
Let $a, b, u, v \in \Sigma^*$ be geodesic words such that $v b =_G a u$.  
Suppose that $u$ has a factorisation $u = u' u'' u'''$ with $|u'| \geq |a| + 2\delta$, $|u''| \geq 4\delta$, and $|u'''| \geq |b|+2\delta$.
Then there exists a factorisation $v = v' v'' v'''$ and geodesic words $c, d$ so that
\begin{itemize}
\item
$|c|, |d| \leq 2\delta$,
\item
$v' c =_G a u'$,
\item
$v'' d =_G c u''$, and 
\item
$v''' b =_G d u'''$.
\end{itemize}
\end{corollary}

\begin{proof}
We prove this with two applications of Lemma~\ref{lem-quad}.  
The first application gives us $c$.  
In the second application we restrict our attention to the quadrilateral with sides labelled $c$, $b$, $u'' \cdot u'''$, and the fourth side labelled by the resulting suffix of $v$.  
This gives $d$. 
\end{proof}

Suppose that $S$ is a path in $\Gamma$ of length $n$, and $i$ is an integer. 
We adopt the convention that the use of the expression $S(i)$ implies that $i$ lies in $[0, n]$.   
Recall that $\bar{S}$ denotes $S$ with its parametrisation reversed.

\begin{lemma}  
\label{lem-thin-triangle}
Let $T$ be a $\delta$--thin geodesic triangle with vertices at $p$, $q$, and $r$ and with sides $P$, $Q$, and $R$.   
Suppose that $P(0) = q = \bar{R}(0)$, $Q(0) = r = \bar{P}(0)$, and $R(0) = p = \bar{Q}(0)$.  
Let $j$ be any integer so that $d_\Gamma(R(j), \bar{Q}(j)) > \delta$.
Then there are integers $i_R < i_Q$ so that $d_\Gamma(R(j), P(i_R)) \leq \delta$ and $d_\Gamma(\bar{Q}(j), P(i_Q)) \leq \delta$.
\end{lemma}

In the statement and the proof we follow the notation of Figure~\ref{fig-triangle}.

\begin{proof}[Proof of Lemma~\ref{lem-thin-triangle}]
Since $d_\Gamma( R(j), \bar{Q}(j) ) > \delta$, the group elements $R(j)$ and $\bar{Q}(j)$ do not correspond to each other.  
Thus $R(j)$ is strictly after the meeting point $c_R$ along $R$.  
Similarly $\bar{Q}(j)$ is strictly after the meeting point $c_Q$ along $\bar{Q}$.  
Since $T$ is $\delta$--thin, there are integers $i_R$ and $i_Q$ so that 
\begin{itemize}
\item
$R(j)$ corresponds to $P(i_R)$ and so $d_\Gamma(R(j), P(i_R)) \leq \delta$ and 
\item
$\bar{Q}(j)$ corresponds to $P(i_Q)$ and so $d_\Gamma(\bar{Q}(j), P(i_Q)) \leq \delta$. 
\end{itemize}
We deduce that $P(i_R)$ is strictly before, and $P(i_Q)$ is strictly after, $c_P$ along $P$.  
Thus $i_R < i_Q$ and we are done. 
\end{proof}

\section{Compressed words and the compressed word problem} 
\label{sec-SLP}

\subsection{Straight-line programs}
\label{sec-SLP-def}

Straight-line programs offer succinct representations of long words that contain many repeated substrings.  
We here review the basics, referring to~\cite{Loh14} for a more in-depth introduction. 

\begin{definition}
Fix $\Sigma$, a finite alphabet.  A \emph{straight-line program} over $\Sigma$ is a triple $\calG = (V, S, \rho)$ where
\begin{itemize}
\item
$V$ is a finite set of \emph{variables}, disjoint from $\Sigma$, 
\item
$S \in V$ is the \emph{start variable}, and 
\item
$\rho : V \to (V \cup \Sigma)^*$ is an \emph{acyclic production mapping}: that is, the relation
\[
\{ (B,A) \in V \times V \st \mbox{$B$ appears in $\rho(A)$} \}
\]
is acyclic. We call $\rho(A)$ the {\em right-hand side} of $A$.
\end{itemize}
\end{definition}

\begin{example}
\label{exam-double}
Let $\Sigma = \{a, b\}$ and fix $n \geq 0$.  
We define $\calG_n = (\{A_0, \ldots, A_n\}, A_n, \rho)$,
where $\rho(A_0) = ab$ and $\rho(A_{i+1}) = A_i A_i$ for $0 \leq i \leq n - 1$.
\end{example}

\begin{definition}
Given a straight-line program $\calG$ as above, we define an \emph{evaluation} function $\eval = \eval_\calG \from (V \cup \Sigma)^* \to \Sigma^*$ as follows.
\begin{itemize}
\item
$\eval(a) = a$ for $a \in \Sigma$,
\item
$\eval(uv) = \eval(u)\eval(v)$ for $uv \in (V \cup \Sigma)^*$, and 
\item
$\eval(A) = \eval(\rho(A))$ for $A \in V$.
\end{itemize}
One proves by a delicate induction that $\eval$ is well defined.  
We finally take $\eval(\calG) = \eval(S)$.  
We call $\eval(\calG)$ the \emph{output} of the program $\calG$.
\end{definition}

In other words, $\calG$ is a context-free grammar that generates exactly one word $\eval(\calG)$ of $\Sigma^*$.  

So, continuing Example~\ref{exam-double} above we have $\eval(A_0) = ab$ and more generally $\eval(A_i) = (ab)^{2^i}$.  
Thus $\eval(\calG_n) = \eval(A_n) = (ab)^{2^n}$.  
So the output has length $2^{n+1}$. 

We say a straight-line program $\calG = (V, S, \rho)$ over $\Sigma$ is \emph{trivial} if $S$ is the only variable and $\rho(S) = \emptyword = \eval(\calG)$.  

We say that a straight-line program is in \emph{Chomsky normal form} if it is either a trivial program or all right-hand sides are of the form $a \in \Sigma$ or $BC$ with $B, C \in V$.
There is a linear-time algorithm that transforms a given straight-line program $\calG$ into a program $\calG'$ in Chomsky normal form with the same output;
see~\cite[Proposition~3.8]{Loh14}.

\begin{definition}
We define the \emph{size} $|\calG|$ of $\calG = (V, S, \rho)$ to be the sum of the bit-lengths of the right-hand sides of $\rho$. Symbols from $V \cup \Sigma$
are encoded by bit strings of length $O(\log (|V|+|\Sigma|))$ using a prefix code.
\end{definition}

Again considering Example~\ref{exam-double}, we see that the size of $\calG_n$ is $O(n \log(n))$.
(Note that we take into account the cost of writing out the indices of the variables $A_i$.)
Thus we see that straight-line programs can achieve (essentially) exponential compression.
The following result proves that straight-line programs can do no better; 
the proof follows the proof of~\cite[Lemma~1]{CLLLPPSS05}.\footnote{In \cite{CLLLPPSS05},
$|\calG|$ is defined as the sum of all lengths of right-hand sides of $\calG$. Note that this value is less
than or equal to our value of $|\calG|$ (the bit-lengths of the right-hand sides).}

\begin{lemma}
\label{lem-SLP-upper-bound}
For every straight-line program $\calG$ we have $|\eval(\calG)| \leq  3^{|\calG|/3}$. \qed
\end{lemma}

As a convenient short-hand, we will refer to straight-line programs over $\Sigma$ as \emph{compressed words}.

\subsection{Algorithms for compressed words}
\label{subsec-algo-compressed}

We will assume that all integers given as input to algorithms are given in binary.
We will need to know that the following algorithmic tasks can be solved in polynomial time; see~\cite[Proposition~3.9]{Loh14}.

Given a straight-line program $\calG$ and natural numbers $i \leq j$: 
\begin{itemize}
\item find the length $|\eval(\calG)|$;
\item find the letter $\eval(\calG)[i]$;
\item find a straight-line program $\calG'$ with $\eval(\calG') = \eval(\calG)[i:j]$.
\end{itemize}

The following proposition is also well-known~\cite[Lemma~2]{CLLLPPSS05}.

\begin{proposition}
\label{prop-power}
There is a polynomial-time algorithm that, given a straight-line program $\calG$ and a natural number $n > 0$, computes a straight-line program $\calG_n$ with $\eval(\calG_n) = \eval(\calG)^n$.  In fact, the time required is linear in $|\calG| + \log n$. \qed
\end{proposition}

The following results are less trivial.
A proof of this proposition can be found in~\cite[Theorem 3.11]{Loh14}.

\begin{proposition}
\label{prop-fsa}
There is a polynomial-time algorithm that, given 
\begin{itemize}
\item
a finite alphabet $\Sigma$, 
\item
a finite state automaton $M$ over $\Sigma$, and 
\item
a straight-line program $\calG$ over $\Sigma$,
\end{itemize}
decides if $\eval(\calG)$ lies in the language $L(M)$. \qed
\end{proposition}

We also need the following variant of Proposition~\ref{prop-fsa}.

\begin{proposition}
\label{prop-fsa-2}
There is a polynomial-time algorithm that, given 
\begin{itemize}
\item
a finite alphabet $\Sigma$, 
\item
a finite state automaton $M$ over $\Sigma$, and 
\item
a straight-line program $\calG$ over $\Sigma$,
\end{itemize}
decides if $\{ \eval(\calG)^n \st n \in \NN \}$ is a subset of $L(M)$. 
\end{proposition}

\begin{proof}
Let $M = (Q, \Sigma, q_0, \delta, F)$ be the automaton.  
Suppose that $w = \eval(\calG)$.
All non-negative powers of $w$ belong to $L(M)$ if and only if $\delta(q_0, w^n)$ lies in $F$, for all $n \geq 0$. 

Since $Q$ is finite, there are natural numbers $k$ and $l$, with $0 \leq k < l \leq |Q|$ such that $\delta(q_0, w^k) = \delta(q_0, w^l)$ and hence
\[
\delta(q_0, w^{k+i}) = \delta(q_0, w^{l+i}) \mbox{ for all } i \ge 0.
\]
It follows that $w^n \in L(M)$ for all $n \geq 0$ if and only if $w^n \in L(M)$ for all $0 \leq n \leq |Q|$.
By Proposition~\ref{prop-power}, we can compute, in polynomial time and for all $0 \leq n \leq |Q|$, a straight-line program $\calG_n$ with output $\eval(\calG_n) = w^n$. 
Finally we use Proposition~\ref{prop-fsa} to test, in polynomial time, if $\eval(\calG_n) \in L(M)$ for these programs. 
\end{proof}

The following result is central to our past and present work. 
It was independently discovered by Hirshfeld, Jerrum, and Moller~\cite[Proposition~12]{HirshfeldJM94} 
(see also \cite[Proposition~3.2]{HiJeMo96}), 
by Mehlhorn, Sundar, and Uhrig~\cite{MehlhornSU94,MehlhornSU97} 
(where the result is implicitly stated in terms of dynamic string data structures), and
by Plandowski~\cite[Theorem~13]{Pla94}. 

\begin{theorem}  
\label{thm-plandowski}
There is a polynomial-time algorithm that, given straight-line programs $\calG$ and $\calH$, decides if $\eval(\calG) = \eval(\calH)$. \qed
\end{theorem}

We now give a version of \cite[Theorem~1]{KarpinskiEtAl95}; this generalises Theorem~\ref{thm-plandowski} to the so-called \emph{fully compressed pattern matching problem}.  
See~\cite[Theorem~1.1]{Jez15} for a quadratic time algorithm, which is the best currently known. 

\begin{theorem}
\label{thm-pattern-matching}
There is a polynomial-time algorithm that, given straight-line programs $\calG$ and $\calH$, decides if $\eval(\calG)$ is a factor of $\eval(\calH)$.  
Furthermore, if it is a factor, the algorithm returns (in binary) the smallest $m \in \NN$ so that $\eval(\calG)$ is a prefix of $\eval(\calH)[n:]$. \qed
\end{theorem}

We obtain the following corollary of Theorem~\ref{thm-pattern-matching} and Lemma~\ref{lem-rotation}.

\begin{corollary}
\label{coro-cyclic-conj}
There is a polynomial-time algorithm that, given straight-line programs $\calG$ and $\calH$, decides if $\eval(\calG)$ is a rotation of $\eval(\calH)$.  
Furthermore, if it is, then the algorithm returns straight-line programs $\calH'$ and $\calH''$ such that
\[
\pushQED{\qed}
\eval(\calH) = \eval(\calH') \eval(\calH'') 
    \qquad \mbox{and} \qquad
\eval(\calG) = \eval(\calH'') \eval(\calH')  \qedhere
\popQED
\]
\end{corollary}

\subsection{The compressed word problem}

Suppose that $G$ is a group and $\Sigma$ is a finite symmetric generating set.  The \emph{compressed word problem} for $G$, over $\Sigma$, is the following decision problem.

\begin{description}
\item [Input] A straight-line program $\calG$ over $\Sigma$.
\item [Question] Does $\eval(\calG)$ represent the identity of $G$?
\end{description}

Note that the compressed word problem for a group $G$ is decidable if and only if the word problem for $G$ is decidable.
As discussed in the introduction, there are in fact groups $G$ where the compressed word problem is strictly harder than the word problem itself. 

Observe that the computational complexity of the compressed word problem for $G$ does not depend on the chosen generating set $\Sigma$. 
That is, if $\Sigma'$ is another such, then the compressed word problem for $G$ over $\Sigma$ is logspace reducible to the compressed word problem for $G$ over $\Sigma'$~\cite[Lemma~4.2]{Loh14}.
Thus, when proving that the compressed word problem is polynomial time, we are allowed to use whatever generating set is most convenient for our purposes.

\begin{remark}
\label{rem-inverse}
As a simple but useful tool, note that if $\calG$ is a straight-line program over $\Sigma$ with output $w$ then there is a straight-line program $\bar{\calG}$ with output $w^{-1}$. 
\end{remark}

\subsection{Cut programs}
\label{sec-cut}

A useful generalisation of straight-line programs are the \emph{composition systems} of~\cite[Definition~8.1.2]{Hag00}.
These are also called \emph{cut straight-line programs} in \cite{Loh14}.  We shall simply call them \emph{cut programs}.
They are used, for example, in the polynomial-time algorithm for the compressed word problem of a free group~\cite{Loh06siam}.  

A \emph{cut program} over $\Sigma$ is a tuple $\calG = (V, S, \rho)$, with $V$ and $S$ as in Section~\ref{sec-SLP-def}, and where we also allow, as right-hand sides for $\rho$, expressions of the form $B[i:j]$, with $B \in V$ and with $i \leq j$.  
We again require $\rho$ to be acyclic.
When $\rho(A) = B[i:j]$ we define 
\[
\eval(A) = \eval(B)[i:j]
\]
with the cut operator $[i:j]$ as defined in Section~\ref{sec-notation}.  
Note that this is only well-defined if $0 \leq i \leq j \leq |\eval(B)|$.
This condition will be always ensured in the rest of the paper.
The \emph{size} of a cut program $\calG$ is the sum of the bit-lengths of the right-hand sides; as usual all natural numbers are written in binary. 

We can now state a straightforward but important result of Hagenah; see~\cite[Algorithmus~8.1.4]{Hag00} as well as~\cite[Theorem~3.14]{Loh14}.

\begin{theorem}
\label{thm-hagenah}
There is a polynomial-time algorithm that, given a cut program $\calG$, finds a straight-line program $\calG'$ such that 
$\eval(\calG) = \eval(\calG')$.
\end{theorem}

Theorems~\ref{thm-plandowski} and~\ref{thm-hagenah} imply that there is a polynomial-time algorithm that, given two cut programs, decides if they have the same output. 

\begin{remark} \label{remark-cut}
In fact, in what follows, we will only ever need the prefix and suffix cut operators $[:j]$ and $[i:]$.  
This is because, when using a word to represent a group element, cancellation appears where two factors meet. 

We also note that iterating the cut operator can be done using arithmetic alone.  That is, the cut variables
\[
B[i:j][k:\ell] \quad \mbox{and} \quad B[i+k:i+\ell]
\]
have the same evaluation.
This ``cut elimination'' is, in some sense, the heart of the proof of Theorem~\ref{thm-hagenah}.
\end{remark}

\section{The compressed word problem for hyperbolic groups}

Suppose that $G$ is a group and $\Sigma$ is a finite symmetric generating set.  
We fix a total order $<$ on $\Sigma$.
Suppose that $G$ is $\delta$--hyperbolic; 
here we take $\delta$ large enough so that all geodesic triangles are $\delta$--thin, and we assume also that $\delta > 0$ is an integer.  
(This assumption is used in Case~\ref{Case:Trim} inside of the proof of Lemma~\ref{lem-TSLP-SLP}.)
In what follows, we take $\zeta = 2 \delta$.
Recall that $\Ball(r)$ is the ball of radius $r$ about $1_G$ in the Cayley graph $\Gamma = \Gamma(G, \Sigma)$. 

\subsection{Tethered programs}
\label{subsec-tethered}

We introduce a new type of program using the \emph{tether} operator.

A \emph{tethered program} over $\Sigma$ is a tuple $\calG = (V, S, \rho)$, with $V$ and $S$ as in Section~\ref{sec-SLP-def}, and where we also allow, as right-hand sides for $\rho$, expressions of the form $B\teth{a, b}$, with $B \in V$ and with $a, b \in \Ball(\zeta)$.
We again require $\rho$ to be acyclic.
If $\rho(A) = B\teth{a, b}$ then we define 
\[
\eval(A) = \slex(a \cdot \eval(B) \cdot b^{-1})
\]
We call the suffix $\teth{a, b}$ a \emph{tether} operator.  
The \emph{size} of a tethered program $\calG$ is the sum of the bit-lengths of the right-hand sides; group elements in $\Ball(\zeta)$ are represented by their shortlex representatives. 

Finally, in a \emph{tether-cut} program $\calG$ over $\Sigma$ we allow right-hand sides which are words from $(V \cup \Sigma)^*$, a cut variable, or a tethered variable.  It is sometimes convenient to allow more complicated right-hand sides of the form 
$\alpha_1 \cdot \alpha_2 \cdots \alpha_k$ where every $\alpha_i$ is either a symbol from $\Sigma$ or a variable $B$ to which 
a sequence of cut- and tether-operators is applied to.
An example of such a right-hand side is
\[
A[:i] \teth{a, b} \cdot a \cdot B[j:] \teth{c, d}  
\]
Note that a right-hand side of the form $(A \cdot B)[i:j]$ or $(A \cdot B)\teth{a,b}$ are not allowed.

Finally, we define the size of a tether-cut program $\calG$ as the sum of the bit-lengths of the right-hand sides.  In what follows we will assume that all programs arising are over a fixed alphabet $\Sigma$. 

\begin{remark}
\label{rem-double-tether}
In what follows we mostly need the prefix and suffix tether operators $\teth{a, 1}$ and $\teth{1, b}$.  
For, suppose that $\Gamma(G, \Sigma)$ is hyperbolic and that $u$ and $v$ are geodesic words.
Let $w =_G uv$ be a geodesic word representing their product.
Then we can describe $w$ (up to bounded Hausdorff distance) by taking a prefix of $u$, tethering the result at the end, concatenating with a short word, and then tethering (at the front) a suffix of $v$.  
See Figure~\ref{fig-cutting-off-peak} below.

We also note that iterating tether operators can be done ``locally''.  
That is, for any $a, a', b, b' \in \Ball(\zeta)$ there are elements $a'', b'' \in \Ball(\zeta)$, elements $x, y \in \Ball(2\zeta)$, and natural numbers $i, j$ so that the expressions
\[
B \teth{a, b} \teth{a',b'} 
\quad \mbox{and} \quad 
x \cdot B[i:j] \teth{a'', b''} \cdot y
\]
have the same evaluation: that is, they represent the same shortlex reduced word. 
See Figure~\ref{fig-iterated-tether}.
Again, this ``tether-elimination'' is, in some sense, the heart of our proof of Lemma~\ref{lem-TSLP-SLP}.
\end{remark}

\begin{figure}
  \centering{
    \scalebox{1}{
      \begin{tikzpicture}
        \tikzstyle{small} = [circle, draw = black, fill = black, inner sep = .15mm]
        \tikzstyle{zero} = [circle, inner sep = 0mm]

        \node[small] (1) {} ;
        \node[small, right = 9cm of 1] (2) {};

        \node[zero,  above = 1.5cm of 1] (1u) {};        
        \node[zero,  above = 1.5cm of 2] (2u) {};        
        \node[zero,  below = 1.5cm of 1] (1d) {};        
        \node[zero,  below = 1.5cm of 2] (2d) {};        

        \node[small, right = 1.5cm of 1u] (3) {};
        \node[small, left  = 1.5cm of 2u] (4) {};
        \node[small, right = 1.5cm of 1d] (5) {};
        \node[small, left  = 1.5cm of 2d] (6) {};
        \node[zero, right  = 1.15cm of 3] (3r) {};
        \node[small, below  = 0.5cm of 3r] (3') {};
        \node[zero, left  = 1.15cm of 4] (4l) {};
        \node[small, below  = 0.5cm of 4l] (4') {};
        \node[zero, right  = 2.05cm of 5] (5r) {};
        \node[small, above  = 0.75cm of 5r] (5') {};
        \node[zero, left  = 2.05cm of 6] (6l) {};
        \node[small, above  = 0.75cm of 6l] (6') {};

        \draw [->] (3) to [out=-30, in=160] node[pos = 0.5, above right = -0.7mm]{$x$} (3');  
	\draw [->] (3') to [out=-20, in=-160] node[pos = 0.5, above = -0.4mm]{$w''$} (4');
	\draw [->] (4') to [out=20, in=-150]
                                             node[pos = 0.5, above left = -0.7mm]{$y$}  (4);
        \draw [->] (1) edge node[below = -0.7mm]{} (2); 
        \draw [->] (5) to [out=30, in=-170] node[pos = 0.5, below right = -0.7mm]{$u'$} (5');  
	\draw [->] (5') to [out=10, in=170] node[pos = 0.5, below = -0.4mm]{$u''$} (6');
	\draw [->] (6') to [out=-10, in=150]
                                             node[pos = 0.6, below left = -0.7mm]{$u'''$}  (6);
        \draw [->] (1) to [out=-30, in=120] node[below left = -0.7mm]{$a$} (5);
        \draw [->] (2) to [out=-150, in=60] node[below right =-0.7mm]{$b$} (6);

        \draw [->] (3) to [out=-120, in=30] node[above left = -0.7mm]{$a'$} (1);
        \draw [->] (4) to [out=-60, in=150] node[above right =-0.7mm]{$b'$} (2);

        \draw [->] (3') to [out=-50, in =105] node[pos = 0.25, below left = -1.0mm]{$a''$} (5');
        \draw [->] (4') to [out=-130, in=75] node[pos = 0.25, below right = -1.0mm]{$b''$} (6');

  \end{tikzpicture}}}
  \caption{The evaluation of $B\teth{a, b}\teth{a', b'}$ agrees with the evaluation of $x \cdot B[i:j]\teth{a'', b''} \cdot y$.  Here we are assuming that $\eval(B) = u = u'u''u'''$ and that $\eval(B[i:j]\teth{a'', b''}) = w''$.}
  \label{fig-iterated-tether}
\end{figure}
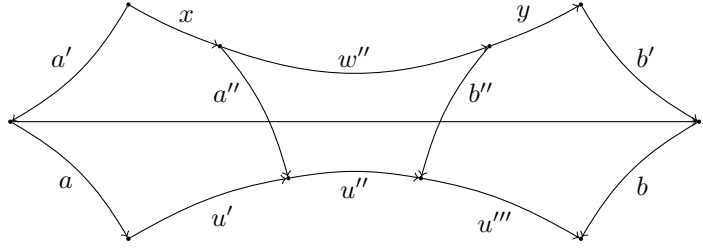

We say that a program is in \emph{Chomsky normal form} if it is either a trivial program or
all right hand sides $\rho(A)$ have one of the following forms, where $B,C \in V$, $a \in \Sigma$, $i \leq j$ and $b, c \in \Ball(\zeta)$:
$a$, $BC$, $B[i:j]$, $B\teth{b,c}$.
Similar to the case of straight-line programs, there is a linear-time algorithm that transforms a given program $\calG$ (with $\eval(\calG) \neq \emptyword$) into a program $\calG'$ in Chomsky normal form with the same output. 

We say that a program $\calG$ is \emph{geodesic} (or \emph{shortlex}) if for every variable $A$, the word $\eval(A)$ is geodesic (shortlex reduced).

\begin{lemma}
\label{lem-CNF}
There is a polynomial-time algorithm that, given a geodesic tether-cut program $\calG$, returns a geodesic tether-cut program $\calG'$ with the same evaluation which is in Chomsky normal form.
\end{lemma}

\begin{proof}
We essentially use the usual algorithm; see, for example~\cite[Proposition~3.8]{Loh14}.  
However, some care must be taken with tethered variables.

By introducing new variables, we can first assume that all right-hand sides of $\calG$ have the form 
$w \in (V \cup \Sigma)^*$, $B[i:j]$ or $B \teth{a, b}$ with $B \in V$. This preserves the property of being
a geodesic tether-cut program, since every factor of a geodesic word is again geodesic.

Next we eliminate variables $B$ with $\rho(B) = \emptyword$. For this we take any variable $B$ with $\rho(B) = \emptyword$,
remove $B$ from the tether-cut program, and replace every
occurrence of $B$ in a right-hand side $\rho(A)$ by the empty word. Note that if $\rho(A) = B[i:j]$ then we must have $i=j=0$
and we set $\rho(A) = \emptyword$. If $\rho(A) = B \teth{a, b}$, then we set $\rho(A) = \slex(a b^{-1}) \in \Sigma^*$ (which has length at most $2\zeta$). Iterating this step will finally eliminate all variables $B$ whose right-hand side
is the empty word. 

The rest of the proof follows the proof of \cite[Proposition~3.8]{Loh14}: by introducing new variables we can assume that
all right-hand sides have one of the following forms: $a \in \Gamma$, $B_1 B_2 \cdots B_n$, $B[i:j]$, $B\teth{b,c}$,
where $B, B_1, \ldots B_n$ are variables and $n \geq 1$. If $\rho(A) = B$ for variables $A, B$ we can remove $A$ and 
replace all occurrences of $A$ in a right-hand by $B$. Iterating this step ensures that whenever $\rho(A) = B_1 B_2 \cdots B_n$ 
for a variable $A$ then $n \geq 3$. Finally, by adding new variables we can split up right-hand sides $B_1 B_2 \cdots B_n$ 
with $n \geq 3$ in right-hand sides consisting of exactly two variables.

Note that the above construction preserves the property of being geodesic, since every factor of a geodesic word is again geodesic.
Also notice that the final program may still have variables $B$ with $\eval(B) = \emptyword$. This is due to the tether operator.
\end{proof}

Note that the concatenation of geodesic words may not itself be geodesic; however the concatenation does provide two sides of a geodesic triangle.  
When the group $G$ is hyperbolic, this gives us the beginnings of a reduction procedure. 

We now turn to the task of proving Proposition~\ref{prop-TCSLP-SLP}.
We will give a sequence of results that allows us to transform a geodesic tether-cut program into a straight-line program, whose evaluation is the shortlex representative of the original.  
The first step, in Lemma~\ref{lem-TSLP-SLP}, gives such a transformation for tethered programs.  
The second step, finishing the proof of Proposition~\ref{prop-TCSLP-SLP}, is to transform a geodesic tether-cut program into an geodesic tethered program with the same output.  
This second step is inspired by Hagenah's result (Theorem~\ref{thm-hagenah}) transforming a cut program into an equivalent straight-line program.

\subsection{Transforming tethered programs}

Suppose that $\calG = (V, S, \rho)$ is a program, as above.  

We recursively define the \emph{height} of elements of $\Sigma \cup V$.  
If $a \in \Sigma$ then we take $\height(a) = 0$.  
For $A \in V$ we define
\[ 
\height(A) = \max \{ \height(B) + 1 \st \mbox{$B \in \Sigma \cup V$ occurs in $\rho(A)$} \}
\]
Finally we set $\height(\calG) = \height(S)$. 

Suppose that $\calG$ is a tethered program in Chomsky normal form.  
If $A \in V$ is a variable we define its \emph{tether-height}, denoted $\height_t(A)$, recursively as follows. 
\begin{itemize}
\item If $\rho(A) = a$, then $\height_t(A) = 0$,
\item if $\rho(A) = BC$, then $\height_t(A)= \max \{ \height_t(B), \height_t(C)\}$, and
\item if $\rho(A) = B \teth{s, t}$ then $\height_t(A) = \height_t(B) + 1$.
\end{itemize}
For a variable $A$ we define its \emph{tether-depth} to be 
\[
\depth_t(A) = \height_t(S) - \height_t(A) + 1
\]

\begin{lemma} 
\label{lem-TSLP-SLP}
There is a polynomial-time algorithm that, given a geodesic tethered program $\calG$, finds a shortlex straight-line program $\calG'$ so that $\eval(\calG') = \slex(\eval(\calG))$. 
\end{lemma}

\begin{proof}
Set $\calG = (V, S, \rho)$.  The straight-line program $\calG'$ that we construct will be of the form $\calG' = (V',S',\rho')$ for suitable $V'$ and $\rho'$.

Applying Lemma~\ref{lem-CNF} we may assume that $\calG$ is in Chomsky normal form.  
Introducing a new start variable, if needed, we may assume that $\rho(S)$ has the form $A \teth{1, 1}$ for a variable $A$.
We do this to force the evaluation of $\calG$ to be shortlex reduced, not just geodesic. 
By removing unused variables, we can assume that $S$ has maximal height and maximal tether-height among all variables.  
This implies that, for all $A \in V$, the tether-depth of $A$ is greater than zero.
Finally, for every variable $A \in V$ such that $\rho(A) = BC$ with $B,C \in V$ we can assume that
\[
\depth_t(A) = \depth_t(B) = \depth_t(C)
\]
To ensure this property we add dummy variables to $\calG$, with productions of the form $X\teth{1, 1}$, as needed.

In the rest of the proof, $\height$, $\height_t$ and $\depth_t$ always refer to the original tethered program $\calG$.

We carry out the proof in a bottom-up fashion; that is, we consider the variables of $\calG$ in order of increasing height.  
Here is an outline of the proof; we give the details below. 
Set $w = \eval(A)$.
If $|w| \leq 16 \zeta \depth_t(A) + 2\zeta$ (such a word will be also called short) then we compute and record $w$, as a word.
If $w > 16 \zeta \depth_t(A) + 2\zeta$  (such a word will be also called long), then we instead compute words $\ell_A$ and $r_A$ such that 
\[
w = \ell_A \cdot w' \cdot r_A
\]
for some word $w'$ of length at least $2 \zeta$.  
The details of the computation depend on the production $\rho(A)$.
We require that the word $\ell_A$ satisfies the following length constraint
\begin{equation}
\label{equ-length}
8 \zeta \depth_t(A) \leq |\ell_A| \le 8 \zeta \depth_t(A) + 2\zeta \height(A)
\end{equation}
and similarly for $r_A$.  

When $w$ is long, we also add to the program $\calG'$ the \emph{decorated} variables $A'_{a,b}$ for all $a, b \in \Ball(\zeta)$.  
We arrange the following:
\[
\eval(A'_{a,b}) = \slex(a \cdot w'  \cdot  b^{-1})
\]
These new variables $A'_{a,b}$, and also a new start variable $S'$, are the \emph{only} variables appearing in $\calG'$, that is, they form the set $V'$.
All of the words that we compute and record along the way, such as the short words $w$ and the prefixes and suffixes $\ell_A$ and $r_A$, are not separately stored as part of $\calG'$.

That completes our outline of the proof.  
We now consider the possibilities for the right-hand side $\rho(A)$.  

\begin{case}
Suppose $\rho(A) \in \Sigma$. 
Thus $w = \eval(A)$ is geodesic and shorter than $16 \zeta \depth_t(A) + 2\zeta$.  We record it and continue. 
\end{case}

\begin{case}
Suppose $\rho(A) = BC$ for variables $B$ and $C$. Recall that we have
\[
\depth_t(A) = \depth_t(B) = \depth_t(C)
\]
Set $\eta = \depth_t(A)$.
Let $u = \eval(B)$, $v = \eval(C)$, and $w = \eval(A) = uv$.
Recall that $u, v, w$ are geodesic by assumption.

\begin{subcase}
\label{sub-both-long}
Suppose $|u| > 16 \zeta \eta + 2\zeta$ and $|v| > 16 \zeta \eta + 2\zeta$.
Hence, in previous stages of the algorithm we computed words $\ell_B, r_B, \ell_C, r_C$ such that the following properties hold.
\begin{itemize}
\item
The prefixes and suffixes $\ell_B, r_B, \ell_C, r_C$ satisfy the length constraint of Equation~\ref{equ-length}.
\item 
There are geodesic words $u', v'$ of length at least $2 \zeta$ with $u = \ell_B \cdot u' \cdot r_B$ and $v = \ell_C \cdot v' \cdot r_C$.
\end{itemize}
Also, we have already defined variables $B'_{a,c}$ and $C'_{d,b}$ for all $a, b, c, d \in \Ball(\zeta)$, which produce $\slex(a \cdot  u' \cdot  c^{-1})$ and $\slex(d \cdot  v' \cdot  b^{-1})$, respectively. 
See Figure~\ref{fig-both-long}.

We now set $\ell_A = \ell_B$ and $r_A = r_C$. 
Since the tether-depths of $A, B, C$ are all the same, but $A$ has greater height, we deduce that $\ell_A$ and $r_A$ satisfy the length constraint of Equation~\ref{equ-length}. 
We also note that 
\[
|u' \cdot r_B \cdot \ell_C \cdot v'| \geq 2\zeta
\]
because $|u'|  \geq 2\zeta$.

It remains to define the right-hand sides for the variables $A'_{a,b}$ for all $a, b \in \Ball(\zeta)$.  
Fix $a, b \in \Ball(\zeta)$. 
For all $c, d \in \Ball(\zeta)$ we compute 
\[
z = \slex( c \cdot r_B \cdot \ell_C \cdot d^{-1} )
\]
in polynomial time using Lemma~\ref{lem-shortlex-reduction}.  
We then check, using Proposition~\ref{prop-fsa} and Lemma~\ref{lem-shortlex-regular}, whether the word
\[
\eval(B'_{a,c}) \cdot z \cdot \eval(C'_{d,b}) 
   = \slex(a \cdot  u' \cdot  c^{-1}) \cdot \slex(c \cdot r_B \cdot \ell_C \cdot d^{-1}) \cdot \slex(d \cdot  v' \cdot  b^{-1}) 
\]
is shortlex reduced, in which case it is equal to 
\[
\slex(a \cdot  u' \cdot r_B \cdot \ell_C \cdot v' \cdot  b^{-1})
\]
Again, see Figure~\ref{fig-both-long}. Since $|u'| \geq 2 \zeta \geq |a|+\zeta = |a|+2\delta$, 
$|v'| \geq 2 \zeta \geq |b|+\zeta = |b|+2\delta$, and $|r_B \ell_C| \geq 16 \zeta \geq 4 \delta$,
Corollary~\ref{cor-quad-two} ensures that there must be at least one such pair $c, d$.  
(If there are several, we stop as soon as we find the first such.)
We then define
\[
\rho'(A'_{a,b}) = B'_{a,c} \cdot z \cdot C'_{d,b}
\]
\end{subcase}

\begin{figure}
  \tikzstyle{vertex}=[circle, fill, inner sep=.5pt, minimum size =.5mm ] 
  \centering{
    \scalebox{1}{
      \begin{tikzpicture}
        \node[vertex] (a) {} ;
        \node[vertex,  right = 1cm of a] (b) {};
        \node[vertex,  right = 4cm of b] (c) {};
        \node[vertex,  right = 1cm of c] (d) {};
        \node[vertex,  right = 1cm of d] (e) {};
        \node[vertex,  right = 4cm of e] (f) {};
        \node[vertex,  right = 1cm of f] (g) {};
    
        \node[vertex,  above = 1cm of b] (b') {};
        \node[vertex,  above = 1cm of c] (c') {};
        \node[vertex,  above = 1cm of e] (e') {};
        \node[vertex,  above = 1cm of f] (f') {};
            
        \draw [->](b') edge node[left=-.7mm]{$a$} (b);
        \draw [<-](c) edge node[left=-.7mm]{$c$} (c');
        \draw [->](e') edge node[right=-.7mm]{$d$} (e);
        \draw [<-](f) edge node[right=-.7mm]{$b$} (f');
     
        \draw [->](a) edge node[below=-.7mm]{$\ell_B$} (b);
        \draw [->, dashed](b) edge node[below=-.7mm]{$u'$} (c);
        \draw [->, dashed](b') edge node[above=-.7mm]{$\slex(a u' c^{-1})$} (c');
        \draw [->](c') edge node[above=-.7mm]{$z$} (e'); 
        \draw [->](c) edge node[below=-.7mm]{$r_B$} (d);
        \draw [->](d) edge node[below=-.7mm]{$\ell_C$} (e);
        \draw [->, dashed](e) edge node[below=-.7mm]{$v'$} (f);
        \draw [->, dashed](e') edge node[above=-.7mm]{$\slex(d v' b^{-1})$} (f');
        \draw [->](f) edge node[below=-.7mm]{$r_C$} (g);
  \end{tikzpicture}}}
  \caption{Case~\ref{sub-both-long} from the proof of Lemma~\ref{lem-TSLP-SLP}.  Dashed lines represent words that are given by straight-line programs. }
  \label{fig-both-long} 
\end{figure}
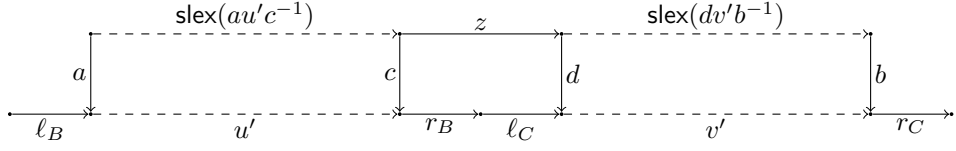

\begin{subcase}
\label{sub-one-long}
Suppose $|u| > 16 \zeta \eta + 2\zeta$ and $|v| \leq 16 \zeta \eta + 2\zeta$.
Thus, at previous stages of the algorithm we computed the geodesic word $v$ explicitly and also computed explicit words $\ell_B$ and $r_B$ such that the following properties hold. 
\begin{itemize}
\item
The prefix and suffix $\ell_B, r_B$ satisfy the length constraint of Equation~\ref{equ-length}.
\item 
There is a geodesic word $u'$ of length at least $2 \zeta$ with $u = \ell_B \cdot u' \cdot r_B$.
\end{itemize}
Also, we already defined variables $B'_{a,b}$ for all $a, b \in \Ball(\zeta)$
such that $B'_{a,b}$ produces $\slex(a \cdot  u' \cdot  b^{-1})$.

If $|v| \leq 2\zeta$, then we set $\ell_A = \ell_B$ and $r_A = r_B v$.  
In this case, we also define $\rho'(A'_{a,b}) = B'_{a,b}$ for all $a, b \in \Ball(\zeta)$.  
Since $\height(B) + 1 \leq \height(A)$, we have the following.
\begin{align*}
8 \zeta \eta  \leq |\ell_A| \leq 8 \zeta \eta + 2\zeta \height(B)       &\leq 8 \zeta \eta + 2\zeta \height(A) \\
8 \zeta \eta  \leq |r_A|    \leq 8 \zeta \eta + 2\zeta (\height(B) + 1) &\leq 8 \zeta \eta + 2\zeta \height(A)
\end{align*}
Thus the length bounds of Equation~\ref{equ-length} are satisfied.

Now assume that $|v| > 2\zeta$. 
Again, we set $\ell_A = \ell_B$.
Since $|r_B \cdot  v| \geq |r_B| \geq 8\zeta \eta$ we can define  $r_A$ as the suffix of $r_B \cdot  v$ of 
length $8\zeta \eta$; that is, $r_B \cdot v = y \cdot r_A$ for some word $y$ of length $|y| = |r_B|+|v|-|r_A| \ge |v| > 2\zeta$. 
This satisfies the required bounds on the lengths of $\ell_A$ and $r_A$.

It remains to define the right-hand sides for the variables $A'_{a,b}$ for all $a, b \in \Ball(\zeta)$.
Let us fix $a, b \in \Ball(\zeta)$. 
For all $c \in \Ball(\zeta)$ we compute $z = \slex(c \cdot  y \cdot  b^{-1})$ and check whether the word
\[
\eval(B'_{a,c}) \cdot z = \slex(a \cdot  u' \cdot  c^{-1}) \cdot \slex(c \cdot  y \cdot  b^{-1}) 
\]
is shortlex reduced. 
If it is, then it equals $\slex(a \cdot  u' \cdot  y \cdot  b^{-1})$; see Figure~\ref{fig-one-long}.
By Lemma~\ref{lem-quad}, there must be at least one such $c$, for which we define
\[
\rho'(A'_{a,b}) = B'_{a,c} \cdot z
\]
\end{subcase}

\begin{figure}
  \tikzstyle{vertex}=[circle, fill, inner sep=.5pt, minimum size =.5mm ] 
  \centering{
    \scalebox{1}{
      \begin{tikzpicture}
        \node[vertex] (a) {} ;
        \node[vertex,  right = 1cm of a] (b) {};
        \node[vertex,  right = 4cm of b] (c) {};
        \node[vertex,  right = 1cm of c] (d) {};
        \node[vertex,  right = 1.5cm of d] (e) {};  
        \node[vertex,  left = 1cm of e] (f) {};
        \node[vertex,  above = 1cm of b] (b') {};
        \node[vertex,  above = 1cm of c] (c') {};
        \node[vertex,  above = 1cm of f] (f') {};
        \draw [->](b') edge node[left=-.7mm]{$a$} (b);
        \draw [<-](c) edge node[left=-.7mm]{$c$} (c');
        \draw [<-](f) edge node[right=-.7mm]{$b$} (f');    
        \draw [->](a) edge node[below=-.7mm]{$\ell_B$} (b);
        \draw [->, dashed](b) edge node[below=-.7mm]{$u'$} (c);
        \draw [->, dashed](b') edge node[above=-.7mm]{$\slex(a u' c^{-1})$} (c');
        \draw [->](c') edge node[above=-.7mm]{$z$} (f'); 
        \draw [->](c) edge node[below=-.7mm]{$r_B$} (d);
        \draw [->](d) edge node[above=-.7mm]{} (e);  
      
        \node[vertex,  inner sep=0pt, minimum size =0mm, above = 1.0mm of c] (c'') {};
        \node[vertex,  inner sep=0pt, minimum size =0mm, above = 1.0mm of f] (f'') {};
        \draw [->](c'') edge node[above=-.7mm]{$y$} (f'');
        
        \node[vertex,  inner sep=0pt, minimum size =0mm, below = 1.0mm of d] (d'') {};
        \node[vertex,  inner sep=0pt, minimum size =0mm, below = 1.0mm of e] (e'') {};
        \draw [->](d'') edge node[below=-.7mm]{$v$} (e'');  
  \end{tikzpicture}}}
  \caption{Case~\ref{sub-one-long} from the proof of Lemma~\ref{lem-TSLP-SLP}.
    Again, dashed lines represent words that are given by straight-line programs.}
  \label{fig-one-long}
\end{figure}
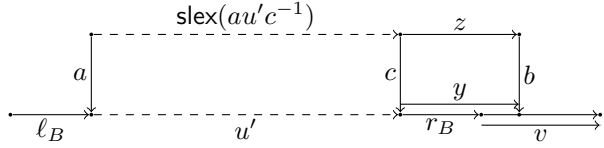

\begin{subcase}
Suppose $|u| \leq 16 \zeta \eta + 2\zeta$ and $|v| > 16 \zeta \eta + 2\zeta$. 
This is dealt with in similar fashion to the previous case. 
\end{subcase}

\begin{subcase}
\label{sub-both-short}
Suppose $|u| \leq 16 \zeta \eta + 2\zeta$ and $|v| \leq 16 \zeta \eta+ 2\zeta$.  
In this case, we have computed $u$ and $v$ explicitly at a previous stage.
We now distinguish between the cases $|w| \leq 16 \zeta \eta + 2\zeta$ and $|w| > 16 \zeta \eta+ 2\zeta$. 
In the first case, we record the word $w$ for later use. 
In the second, we factorise $w$ as $w = \ell_A \cdot w' \cdot r_A$ with $|\ell_A| = |r_A| = 8 \zeta \eta$, and thus $|w'| \ge 2\zeta$. 
We then explicitly compute, for each $a, b \in \Ball(\zeta)$, the word $\slex(a \cdot  w' \cdot  b^{-1})$ and set $\rho'(A'_{a,b})$ equal to it.
This again uses Lemma~\ref{lem-shortlex-reduction}.
\end{subcase}
\end{case}

\begin{case}
\label{case-tether-elim}
Suppose $\rho(A) = B \teth{a, b}$ for $a, b \in \Ball(\zeta)$.
Let $u = \eval(B)$ and $v = \eval(A) = \slex(a \cdot  u \cdot  b^{-1})$.
The word $u$ is geodesic by assumption, and $v$ is shortlex reduced by definition.
Let $\eta = \depth_t(B)$. 
We have $\depth_t(A) = \eta - 1 \geq 1$.

\begin{subcase}
Suppose $|u| \leq 16 \zeta \eta + 2\zeta$. 
Hence, at a previous stage, we explicitly computed the word $u$.
Using Lemma~\ref{lem-shortlex-reduction} we explicitly compute the word $v = \slex(a \cdot  u \cdot  b^{-1})$.  
The rest of the work divides into cases as $|v|$ is less than or equal to $16 \zeta \eta + 2\zeta$ or is greater.  
This is analogous to Case~\ref{sub-both-short} (where $w$ plays the role of $v$). 
\end{subcase}

\begin{subcase}
\label{Case:Trim}
Suppose $|u| > 16 \zeta \eta + 2\zeta$.  
At a previous stage we computed words $\ell_B, r_B$ with the following properties. 
\begin{itemize}
\item
The prefix and suffix $\ell_B, r_B$ satisfy the length constraint of Equation~\ref{equ-length}.
\item 
There is a geodesic word $u'$ of length at least $2 \zeta$ with $u = \ell_B \cdot u' \cdot r_B$.
\end{itemize}
Also, we already defined variables $B'_{c,d}$ for all $c, d \in \Ball(\zeta)$
such that $B'_{c,d}$ produces $\slex(c u' d^{-1})$.

We check for all $c,d \in \Ball(\zeta)$ whether 
\begin{align*}
& \slex(a \cdot \ell_B \cdot c^{-1}) \cdot \eval(B'_{c,d}) \cdot \slex(d \cdot r_B \cdot b^{-1}) \\
     = \ & \slex(a \cdot \ell_B \cdot c^{-1}) \cdot \slex(c \cdot u' \cdot d^{-1}) \cdot \slex(d \cdot r_B \cdot b^{-1}) 
\end{align*}
is shortlex reduced.  
If it is shortlex reduced then it equals 
\[
\slex(a \cdot \ell_B \cdot u' \cdot r_B \cdot b^{-1}) = \slex(a \cdot  u \cdot b^{-1}) = v
\]
See Figure~\ref{fig-tether-elim}.
By Corollary~\ref{cor-quad-two}, there must exist such $c, d \in \Ball(\zeta)$.
Let $v' = \eval(B'_{c,d}) = \slex(c u' d^{-1})$.

Let $s = \slex(a \cdot \ell_B \cdot c^{-1})$ and $t = \slex(d \cdot r_B \cdot b^{-1})$. 
By the triangle inequality, these words have length at least $8 \zeta \eta - 2 \zeta$.
Hence we can factorise these words as $s = wx$ and $t = yz$ with 
\[
|w| = |z| = 8 \zeta (\eta - 1) = 8\zeta \depth_t(A) \geq 8\zeta
\]
Again, see Figure~\ref{fig-tether-elim}.
The words $x$ and $y$ have length at least $6\zeta$. 
We set $\ell_A = w$ and $r_A = z$.
These words satisfy the required bounds on their lengths.
Note that 
\begin{eqnarray*}
&& \eval(A) = \slex(a \cdot u \cdot b^{-1}) = \ell_A \cdot x \cdot v' \cdot y \cdot r_A
\quad \mbox{and} \quad \\
&& |x \cdot v' \cdot y| \geq 12\zeta \geq 2 \zeta
\end{eqnarray*}
It remains to define the right-hand sides of the variables $A'_{a', b'}$ for all $a', b' \in \Ball(\zeta)$.
(This, in essence, is where we call upon Remark~\ref{rem-double-tether}.)

Fix $a', b' \in  \Ball(\zeta)$.
The lower bounds on the lengths of $w, x, y, z$ allow us to apply Lemma~\ref{lem-quad} to the geodesic quadrilaterals with sides labelled $a, \ell_B, c, wx$ and $d, r_B, b, yz$, respectively.  
Note that all of these words have been computed explicitly. 
Applying Lemma~\ref{lem-shortlex-reduction}, we compute in polynomial time words $e, f \in \Ball(\zeta)$ and factorisations $\ell_B = w' x'$ and $r_B = y' z'$
such that $a w' =_G w e$, $e x' =_G x c$, $d y' =_G y f$, and $f z' =_G z b$.
Once again, see Figure~\ref{fig-tether-elim}.
Now consider the geodesic quadrilateral with sides labelled $x' \cdot u' \cdot y'$, $\slex(a' e)$, $\slex(b' f)$, and $\slex(a' e \cdot x' \cdot u' \cdot y' \cdot (b' f)^{-1})$.  
The triangle inequality implies
$|x'|, |y'| \geq 4\zeta$ and $|\slex(a' e)|, |\slex(b' f)| \leq 2\zeta$. 
Again applying Corollary~\ref{cor-quad-two} there are $c', d' \in \Ball(\zeta)$ such that the word
\begin{align*}
& \slex(a' e \cdot x' \cdot (c')^{-1}) \cdot \eval_{\calG'}(B'_{c', d'}) \cdot \slex(d' \cdot y' \cdot (b'f)^{-1}) \\
 = \   & \slex(a' e \cdot x' \cdot (c')^{-1}) \cdot \slex(c' \cdot u' \cdot (d')^{-1}) \cdot \slex(d' \cdot y' \cdot (b' f)^{-1}) 
\end{align*}
is shortlex reduced.  
Thus the above word is 
\[ \slex(a' e \cdot x'  \cdot u' \cdot  y' \cdot (b' f)^{-1}) = \slex(a'  \cdot x  \cdot v' \cdot  y' \cdot (b')^{-1}) \]
As before, we can compute such $c', d' \in  \Ball(\zeta)$ in polynomial time.
We finally define the right-hand side of $A'_{a', b'}$ as
\[
\rho'(A'_{a', b'}) = \slex(a' e \cdot x' \cdot (c')^{-1}) \cdot B'_{c',d'} \cdot \slex(d' \cdot y' \cdot (b' f)^{-1}) 
\]
This concludes the definition of the right-hand sides for the variables $A'_{a', b'}$.
\end{subcase}
\end{case}

We complete the definition of the straight-line program $\calG'$.  
We add a new start variable $S'$ to $\calG'$. 
If $\eval(S)$ is short then we set $\rho'(S') = \eval(S)$ and we are done.
If $\eval(S)$ is long then we set $\rho'(S') = \ell_S \cdot S'_{1, 1} \cdot r_S$.
This ensures $\eval(\calG') = \ell_S \cdot \slex(s') \cdot r_S$, where $s'$ is such that $\ell_S \cdot s' \cdot r_S = \eval(S) = \eval(\calG)$. 
But $\eval(\calG)$ is shortlex reduced (since $\rho(S)$ has the form $A \teth{1, 1}$).
Hence $s'$ is also shortlex reduced and we find 
\[
\eval(\calG') = \ell_S \cdot \slex(s') \cdot r_S = \ell_S \cdot s' \cdot r_S = \eval(\calG)
\]
This concludes the proof of the lemma.
\end{proof}

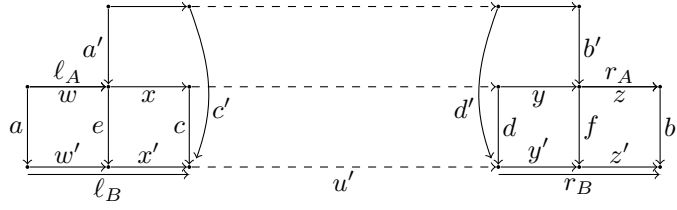
\begin{figure}[t]
  \tikzstyle{vertex}=[circle, fill, inner sep=.5pt, minimum size =.5mm ] 
  \tikzstyle{zero} = [circle, inner sep=0mm]
  \centering{
    \scalebox{1}{
      \begin{tikzpicture}
        \node[vertex] (a) {};
        \node[vertex,  right = 1cm of a] (b) {};
        \node[vertex,  right = 1cm of b] (c) {};
        \node[vertex,  right = 4cm of c] (d) {};
        \node[vertex,  right = 1cm of d] (e) {};  
        \node[vertex, right = 1cm of e] (f) {};
        
        \node[vertex,  above = 1cm of a] (a') {};
        \node[vertex,  above = 1cm of b] (b') {};
        \node[vertex,  above = 1cm of c] (c') {};
        \node[vertex,  above = 1cm of d] (d') {};
        \node[vertex,  above = 1cm of e] (e') {};
        \node[vertex,  above = 1cm of f] (f') {};
      
        \node[vertex,  above = 1cm of b'] (b'') {};
        \node[vertex,  above = 1cm of c'] (c'') {};
        \node[vertex,  above = 1cm of d'] (d'') {};
        \node[vertex,  above = 1cm of e'] (e'') {};

        \node[zero,  above = 1pt of c] (ca) {};
        \node[zero,  right = 1pt of ca] (car) {};
        \node[zero,  above = 1pt of d] (da) {};
        \node[zero,  left = 1pt of da] (dal) {};
           
        \draw [->](a) edge node[above=-.7mm]{$w'$} (b);
        \draw [->](b) edge node[above=-.7mm]{$x'$} (c);
        \draw [->, dashed](c) edge node[below=-.7mm]{$u'$} (d);
        \draw [->](d) edge node[above=-.7mm]{$y'$} (e);
        \draw [->](e) edge node[above=-.7mm]{$z'$} (f);
     
        \draw [->](a') edge node[below=-.7mm]{$w$} (b');
        \draw [->](a') edge node[above=-.7mm]{$\ell_A$} (b');
        \draw [->](b') edge node[below=-.7mm]{$x$} (c');
        \draw [->, dashed](c') edge node[below=-.7mm]{$v'$} (d');
        \draw [->](d') edge node[below=-.7mm]{$y$} (e');
        \draw [->](e') edge node[below=-.7mm]{$z$} (f');
        \draw [->](e') edge node[above=-.7mm]{$r_A$} (f');
          
        \draw [->](b'') edge node[below]{} (c'');
        \draw [->, dashed](c'') edge node[below]{} (d'');
        \draw [->](d'') edge node[below]{} (e'');
        
        \draw [->](a') edge node[left=-.7mm]{$a$} (a);
        \draw [->](b') edge node[left=-.7mm]{$e$} (b);
        \draw [->](c') edge node[left=-.7mm]{$c$} (c);
        \draw [<-](d) edge node[right=-.7mm]{$d$} (d');
        \draw [<-](e) edge node[right=-.7mm]{$f$} (e');
        \draw [<-](f) edge node[right=-.7mm]{$b$} (f');
      
        \draw [->](b'') edge node[left=-.7mm]{$a'$} (b');
        \draw [<-](e') edge node[right=-.7mm]{$b'$} (e'');
      
        \draw [->](c'') to[out=-70,in=70,pos=0.67] node[right=-.7mm] {$c'$} (car);
	\draw [<-](dal) to[out=-250,in=250,pos=0.33] node[left=-.7mm] {$d'$} (d'');
       
        \node[vertex,  inner sep=0pt, minimum size =0mm, below = .6mm of a] (a''') {};
        \node[vertex,  inner sep=0pt, minimum size =0mm, below = .6mm of c] (c''') {};
        \draw [->](a''') edge node[below=-.7mm]{$\ell_B$} (c''');  
       
        \node[vertex,  inner sep=0pt, minimum size =0mm, below = .6mm of d] (d''') {};
        \node[vertex,  inner sep=0pt, minimum size =0mm, below = .6mm of f] (f''') {};
        \draw [->](d''') edge node[below=-.7mm]{$r_B$} (f''');  
  \end{tikzpicture}}}
  \caption{ Case~\ref{case-tether-elim} from the proof of Lemma~\ref{lem-TSLP-SLP}.  
  \label{fig-tether-elim}
Again, dashed lines represent words that are given by straight-line programs. }
\end{figure}

The next lemma follows from Lemma~\ref{lem-TSLP-SLP}.

\begin{lemma} 
\label{lem-TSLP-lengths}
There is a polynomial-time algorithm that, given a geodesic tethered program $\calG$, computes for every $A$ the length
$|\eval(A)|$.
\end{lemma}

\begin{proof}
By Lemma~\ref{lem-TSLP-SLP}, we can compute for every variable $A$ a straight-line program $\calG_A$ with $\eval(\calG_A) = \slex(\eval(A))$.
As in Section~\ref{subsec-algo-compressed}, we can compute $|\eval(\calG_A)| = |\slex(\eval(A))| = |\eval(A)|$ in polynomial time.  
Here the last equality holds since $\calG$ is a geodesic program.
\end{proof}

We now can prove our proposition; this generalises Lemma~\ref{lem-TSLP-SLP} to tether-cut programs.

\begin{proposition} 
\label{prop-TCSLP-SLP}
There is a polynomial-time algorithm that, given a geodesic tether-cut program $\calG$, computes a shortlex straight-line program $\calG'$ such that $\eval(\calG') = \slex(\eval(\calG))$.
\end{proposition}

\begin{proof}
The idea of the proof is taken from the proof of Theorem~\ref{thm-hagenah}; 
see~\cite[Algorithmus~8.1.4]{Hag00}.  
That is, we will eliminate cut operators by pushing them towards smaller variables.  
We then appeal to Lemma~\ref{lem-TSLP-SLP} to eliminate tether operators. 

Let $\calG = (V, S, \rho)$ be the input geodesic tether-cut program.  
By Lemma~\ref{lem-CNF}, we can assume that $\calG$ is in Chomsky normal form. 
Let $\mu = \height(\calG)$. By Lemma~\ref{lem-TSLP-SLP} it suffices to transform $\calG$ into
a geodesic tethered program for $\eval(\calG)$.

We will only consider cuts of the form $[:i]$ and $[i:]$; see Remark~\ref{remark-cut}.
It is not difficult to include also general cuts of the form $[i:j]$.

Consider a variable $A$ such that $\rho(A) = B[:i]$; the case that $\rho(A) = B[i:]$ is dealt with analogously.
We consider variables in order of \emph{increasing} height; so the algorithm is bottom-up.  
By induction we may assume that no cut operator occurs in the right-hand side of any variable $C$ with height less than that of $A$. 

We now must eliminate the cut operator in $\rho(A)$. 
In so doing, we add at most $\mu$ new variables to the tether-cut program. 
Moreover the height of the tether-cut program after the cut elimination will still be bounded by $\mu$. 
Hence, the final tethered program will have at most $\mu \cdot |V|$ variables. 
In addition, the bit length of every  new right-hand side will be polynomially bounded in the input length.
Thus, the size of the  final tethered program will be polynomially bounded in the input length.

Recall that $\rho(A) = B[:i]$. 
We divide the work into cases, depending on the form of $\rho(B)$.  
Since we already have processed $B$, only one of the following cases can occur.

\begin{case_t}
Suppose $\rho(B) = a \in \Sigma$.
If $i = 1$ we redefine $\rho(A) = a$, and if $i = 0$ we redefine $\rho(A) = \emptyword$.
\end{case_t}

\begin{case_t}
\label{case-concat}
Suppose $\rho(B) = CD$ with $C, D \in V$.  
We compute $n_C = |\eval(C)|$ using Lemma~\ref{lem-TSLP-lengths} and with an appeal to the induction hypothesis. 
If $i \leq n_C$ then we redefine $\rho(A) = C[:i]$.
If $i > n_C$ then we add a new variable $X$, we define $\rho(X) = D[:i - n_C]$, and we redefine $\rho(A) = CX$.
We then eliminate the new cut operator in $C[:i]$ or in $D[:i-n_C]$ with a top-down sub-routine.  
(This is what leads to the quadratic growth of new variables.)
\end{case_t}

\begin{case_t}
Suppose $\rho(B) = C\teth{a, b}$ with $C \in V$ and $a, b \in \Ball(\zeta)$.
Let $u = \eval_{\calG}(C)$, $v = \eval_{\calG}(B)$, and $v = v' v''$ with $|v'| = i$. 
Thus, we have $\eval_{\calG}(A) = v'$ and $v = \slex(a u b^{-1})$. 
By Lemma~\ref{lem-TSLP-SLP} we can assume that we have straight-line programs for the words $u$ and $v$.

\begin{subcase_t}
\label{sub-near-front}
There exists $c \in \Ball(\zeta)$ and a factorisation $a = a' a''$ such that $v'  =_G a'c$; see Figure~\ref{fig-near-front}. 
Note that this implies that $i = |v'| \leq 2\zeta$.
Hence, we can check in polynomial time whether this case holds by computing the prefix of the compressed word $v$ of length $i$. 
We redefine $\rho(A) = v'$.
\end{subcase_t}

\begin{figure}
  \tikzstyle{small} = [circle,draw=black,fill=black,inner sep=.15mm]
  \tikzstyle{zero} = [circle,inner sep=0mm]
  \centering{
    \scalebox{1}{
      \begin{tikzpicture}
        
        \node[small] (1) {};
        \node[small,  right = 6cm of 1] (2) {};
        \node[small,  above = 1.5cm of 1] (3) {};
        \node[small,  right = 6cm of 3] (4) {};
        
        \node[zero,  above = .6mm of 3] (3') {};
        \node[zero,  above = .6mm of 4] (4') {};
        
        \draw [->] (3) to [out=-30, in=-150] node[pos=.08,below=-.5mm]{$v'$} 
                                             node[pos=0.2, small] (b) {}
                                             node[pos=.6,below=-.7mm]{$v''$} (4);

        \draw [->](3') to [out=-30, in=-150] node[above=-.7mm]{$v$} (4'); 
     
        \draw [->] (3) to node[pos=.25,left=-.7mm]{$a'$} 
                          node[pos=0.5, small] (a) {}
                          node[pos=.75,left=-.7mm]{$a''$} (1);
        
        \draw [<-](2) edge node[right=-.7mm]{$b$} (4);
        \draw [->](1) edge node[above=-.7mm]{$u$} (2);
        \draw [->](a) to [out=0, in=-150] node[below=-.7mm]{$c$} (b);
  \end{tikzpicture}}}
  \caption{ Case~\ref{sub-near-front} in the proof of Proposition~\ref{prop-TCSLP-SLP}. }
  \label{fig-near-front} 
\end{figure}

\begin{subcase_t}
\label{sub-near-back}
There exists $c \in \Ball(\zeta)$ and a factorisation $b = b'' b'$ such that $v'' b'' =_G c$; see Figure~\ref{fig-near-back}. 
As in Case~\ref{sub-near-front} we can check in polynomial time whether this condition holds. 
We introduce a new variable $X$, we set $\rho(X) = C \teth{1, b'}$, and we redefine $\rho(A) = X \teth{a, c}$.
\end{subcase_t}

\begin{figure}
  \centering{
    \scalebox{1}{
      \begin{tikzpicture}
        \tikzstyle{small} = [circle,draw=black,fill=black,inner sep=.15mm]
        \tikzstyle{zero} = [circle,inner sep=0mm]

        \node[small] (1) {} ;
        \node[small, right = 6cm of 1] (2) {};
        \node[small, above = 1.5cm of 1] (3) {};
        \node[small, right = 6cm of 3] (4) {};
    
        \node[zero, above = .6mm of 3] (3') {};
        \node[zero, above = .6mm of 4] (4') {};
      
        \draw [->] (3) to [out=-30, in=-150] node[pos=.4,below=-.7mm] {$v'$} node[pos=0.8, small] (b) {}  node[pos=.92,below=-.8mm]  {$v''$}  (4);
        \draw [->](3') to [out=-30, in=-150] node[above=-.7mm]{$v$} (4'); 
     
        \draw [->](3) edge node[left=-.7mm]{$a$} (1);
        \draw [<-] (2) to node[pos=.25,right=-.7mm] {$b'$} node[pos=0.5, small] (a) {}  node[pos=.75,right=-.7mm]  {$b''$}  (4);
        
        \draw [->](1) edge node[above=-.7mm]{$u$} (2);
        \draw [->](b) to [out=-30, in=180] node[below=-.7mm]{$c$} (a);
  \end{tikzpicture}}}
  \caption{Case~\ref{sub-near-back} in the proof of
    Proposition~\ref{prop-TCSLP-SLP}.}
  \label{fig-near-back} 
\end{figure}
  
\begin{subcase_t}
Neither Case~\ref{sub-near-front} nor Case~\ref{sub-near-back} holds. 
In this case, there exists a factorisation $u = u' u''$ and $c \in \Ball(\zeta)$ such that $v' c =_G a u'$ and $v'' b =_G c u''$; see Figure~\ref{fig-quad}.
The triangle inequality implies $i - 2 \zeta \leq |u'| \leq i + 2 \zeta$. 
We can find such a factorisation of $u$ in polynomial time; we note that $j = |u'|$ lies in $\NN$ and satisfies $|i - j| \leq 2\zeta$.  
So, using Theorem~\ref{thm-hagenah}, we find straight-line programs for the $4\zeta+1$ many words $u' = u[:j]$, where $j \in \NN$, $|i - j| \leq 2\zeta$.
Since $u$  is geodesic, also all factors of $u$ are geodesic. Hence, the straight-line programs for the words $u' = u[:j]$ must be geodesic too.
Then we apply Lemma~\ref{lem-TSLP-SLP} and compute for every $c \in \Ball(\zeta)$ a shortlex straight-line program for the word $w' = \slex(a u' c^{-1})$.  
Theorem~\ref{thm-hagenah} yields a shortlex straight-line  program for $v' = v[:i]$.  
Finally, we check, using Theorem~\ref{thm-plandowski}, whether $v' = w'$.

Hyperbolicity ensures that we will find at least one such $j$ and $c$.  
We introduce a new variable $X$, we set $\rho(X) = C[:j]$, and we redefine $\rho(A) = X \teth{a, c}$ 
We then continue with the elimination of the cut operator in $C[:j]$, as in Case~\ref{case-concat}.
\end{subcase_t}

\end{case_t}
This concludes the proof of the lemma.
\end{proof}

Recall our convention: if $w \in \Sigma^*$ is a word then $g_w \in G$ is the corresponding group element.  
Thus $g_w$ is a vertex of the Cayley graph $\Gamma = \Gamma(G, \Sigma)$. 

\begin{lemma} 
\label{lem-small-dist}
There is a polynomial-time algorithm that, given geodesic tether-cut programs $\calG$ and $\calH$, determines if $d_\Gamma(g, h) \leq \delta$, where $g = g_{\eval(\calG)}$ and $h = g_{\eval(\calH)}$.
Moreover, when this holds, the algorithm also finds an element $b \in \Ball(\delta)$ such that $g =_G h b$.
\end{lemma}

\begin{proof}
Let $S$ and $T$ be the start variables of $\calG$ and $\calH$, respectively.
For all $b \in \Ball(\delta)$ we produce a new geodesic tether-cut program $\calG^b$ for $\slex( \eval(\calG) b^{-1})$.  
We do this by adding to $\calG$ a new start variable with right-hand side $S \teth{1, b}$.

We also add to $\calH$ a new start variable with right-hand side $T \teth{1,1}$ and denote the resulting 
tether-cut program by $\calH^1$.
This ensures that the evaluation of $\calH^1$ is $\slex( \eval(\calH) )$. 
Using Proposition~\ref{prop-TCSLP-SLP} and Theorem~\ref{thm-plandowski} we now check, in polynomial time, if $\eval(\calG^b) = \eval(\calH^1)$.  
This is equivalent to $g =_G h b$.
\end{proof}

\subsection{Solving the compressed word problem}

We now prove our main result.  
Recall that $\Sigma$ is a symmetric generating set for the hyperbolic group $G$.  

\begin{theorem} 
\label{thm-SLP-for-shortlex}
There is a polynomial-time algorithm that, given a straight-line program $\calG$ over $\Sigma$, finds a straight-line program $\calH$ with evaluation $\slex(\eval(\calG))$.
\end{theorem}

\begin{proof}
By Proposition~\ref{prop-TCSLP-SLP} it suffices to build, in polynomial time, a geodesic tether-cut program $\calH$ for $\slex(\eval(\calG))$. 
We process $\calG$ from the bottom-up; that is, we consider its variables in order of increasing height. 
Set $\calG = (V, S, \rho)$; applying~\cite[Proposition~3.8]{Loh14}, we may assume that $\calG$ is in Chomsky normal form.
We build by induction on the height a new tether-cut program $\calG' = (V', S', \rho')$ over $\Sigma$; here $V' = \{ A' \mid A \in V\}$ is a copy of $V$ and $S' \in V'$ is the variable corresponding to $S$.  
The construction will ensure that $\eval(A') = \slex(\eval(A))$ for every $A \in V$.

In the base cases we have $\rho(A) = a \in \Sigma$.  Here we set $\rho'(A') = \slex(a)$.  

In the inductive step we have $\rho(A) = BC$.  
Since $B$ and $C$ have smaller height than $A$ they satisfy the induction hypotheses. 
Set 
\[
u = \slex(\eval(B)) = \eval(B')
\quad \mbox{and} \quad
v = \slex(\eval(C)) = \eval(C')
\]
By Proposition~\ref{prop-TCSLP-SLP} we can transform the geodesic tether-cut programs with start variables $B'$ and $C'$ into shortlex straight-line programs. 
Using these, we compute the lengths $m = |u|$ and $n = |v|$.  
If one or both of these have length zero then we accordingly take $\rho'(A') = C'$ or $\rho'(A') = B'$ or $\rho'(A') = \emptyword$. 
We now assume that $m$ and $n$ are both non-zero.  
Breaking symmetry, we assume that  $m \leq n$.

Let $P$ be the path in the Cayley graph $\Gamma$ starting at $1_G$, ending at $u$, and labelled by $u$.  Similarly, let $Q$ be the path starting at $u$, ending at $uv$, and labelled by $v$.  Finally, let $R$ be the path starting at $1_G$, ending at $uv$, and labelled by $\slex(uv)$. 
See Figure~\ref{fig-cutting-off-peak}.  The path $\bar{P}$, the reverse of $P$, is labelled by $u^{-1}$.  Applying Remark~\ref{rem-inverse} we invert the geodesic straight-line program for $u$ to give a straight-line program for $u^{-1}$.  Using Lemma~\ref{lem-small-dist} we can check whether or not
\[
d_\Gamma(\bar{P}(m), Q(m)) \leq \delta
\]
We break into cases accordingly. 

\begin{case_s}
\label{case-end-close}
Suppose that $d_\Gamma(\bar{P}(m), Q(m)) \leq \delta$. 
We compute, again using Lemma~\ref{lem-small-dist}, a word $a$ of length at most $\delta$ such that $a =_G u v[:m]$.
See the left-hand side of Figure~\ref{fig-cutting-off-peak}.  
In this case we set $\rho'(A') = C'[m:] \teth{a, 1}$.
\end{case_s}

\begin{case_s}
\label{case-end-far}
Suppose that $d_\Gamma(\bar{P}(m), Q(m)) > \delta$. 
Using binary search, we compute an integer $k \in [0, m - 1]$ such that
\[
d_\Gamma(\bar{P}(k), Q(k)) \leq \delta
\quad \mbox{and} \quad
d_\Gamma(\bar{P}(k + 1), Q(k + 1)) > \delta
\]
Here are the details of the binary search.  
We store an interval $[p, q] \subseteq [0, m]$ such that 
\begin{itemize}
\item
$p < q$, 
\item
$d_\Gamma(\bar{P}(p), Q(p)) \leq \delta$, and
\item
$d_\Gamma(\bar{P}(q), Q(q)) > \delta$.
\end{itemize}
We begin with $p = 0$ and $q = m$.  
We stop when $q = p + 1$. 
In each iteration, we compute $r = \ceiling{(p + q)/2}$ and check, using Lemma~\ref{lem-small-dist}, whether 
\[
d_\Gamma(\bar{P}(r), Q(r)) \leq \delta
\quad \mbox{or} \quad
d_\Gamma(\bar{P}(r), Q(r)) > \delta
\]
In the first case we set $p = r$ and do not change $q$; in the second case we set $q = r$ and do not change $p$. 
In each iteration the size of the interval $[p, q]$ is roughly halved. 
Thus the binary search halts after $O(\log(m))$ iterations; this is polynomial in the input size. 
In addition to the final position $k$, we record a word $a \in \Ball(\delta)$ that labels a path from $\bar{P}(k)$ to $Q(k)$.  
Let $j = k + 1$. 

Recall that $R$ is the path from $\bar{P}(m) = 1_G$ to $Q(n)$ labelled by $\slex(uv)$. 
By Lemma~\ref{lem-thin-triangle} there exist $i_P \leq i_Q$ such that 
\[
d_\Gamma( \bar{P}(j), R(i_P)) \leq \delta
\quad \mbox{and} \quad
d_\Gamma( Q(j), R(i_Q)) \leq \delta
\]
For all pairs $b, c \in \Ball(\delta)$ we explicitly compute the word 
\[ 
s = \slex( b \cdot u[m - j] \cdot a \cdot v[k] \cdot c^{-1}) 
\]
The symbols $u[m - j]$ and $v[k]$ can be computed in polynomial time from the available straight-line programs for $u$ and $v$ using
\cite[Proposition~3.9]{Loh14}; see Section~\ref{subsec-algo-compressed}.
For each of these words $s$, we must check if the word
\begin{equation} 
\label{lower-triangle-side}
\slex(u[ : m - j] \cdot b^{-1}) \cdot s \cdot \slex( c \cdot v[j : ] )
\end{equation}
is shortlex reduced; if so, it equals $\slex(uv)$.  
This step can be done using Lemma~\ref{lem-shortlex-regular} and using the given geodesic tether-cut programs for $u$ and $v$.  
From these we obtain geodesic tether-cut programs for $\slex( u[ : m - j] \cdot b^{-1})$ and $\slex( c \cdot v[j : ] )$.  We then use Proposition~\ref{prop-TCSLP-SLP} to transform these into equivalent straight-line programs.

Lemma~\ref{lem-thin-triangle} ensures that that we will find a pair $b, c \in \Ball(\delta)$ such that the word in \eqref{lower-triangle-side} is shortlex reduced.  
Using the first such pair we find, we set  
\[
\rho'(A') = ( B'[ : m - j] \teth{1, b}) \cdot s \cdot (C' [j : ] \teth{c, 1} )
\]
\end{case_s}

This concludes the proof of the theorem.
\end{proof}

\begin{figure}
  \centering{
    \scalebox{1}{
      \begin{tikzpicture}
        \tikzstyle{small} = [circle,draw=black,fill=black,inner sep=.15mm]
        \tikzstyle{zero} = [circle,inner sep=0mm]

        \node[small] (1) {} ;
        \node[zero, above = 0.3mm of 1] (1') {} ;
        \node[small,  right = 1.7cm of 1] (1a) {} ;
        \node[small,  right = 2.3cm of 1] (1b) {} ;
        \node[small,  right = 4cm of 1] (2) {} ;
        \node[zero, above = 0.3mm of 2] (2') {} ;
        \node[small,  above right = 2.5cm and 2cm of 1] (3) {} ;
        \node[zero, left = 0.3mm of 3] (3') {} ;
        \node[zero, right = 0.3mm of 3] (3'') {} ;
    
        \draw [->] (1) to (1a) ;
        \draw [->] (1a) to node[below=-.7mm] {$s$} (1b) ;
        \draw [->] (1b) to (2) ;
    
        \draw [-] (1) to [out=5, in=-95] node[pos=0.45, small] (2a) {}  node[pos=0.55, small] (2b) {}  (3);
        \draw [->] (1') to [out=5, in=-95] node[above left =-.7mm and -.7mm] {$u$} (3');
        \draw [-] (3) to [out=-85, in=175] node[pos=0.45, small] (3a) {}  node[pos=0.55, small] (3b) {}  (2);
        \draw [->] (3'') to [out=-85, in=175] node[above right =-.7mm and -.7mm] {$v$} (2');
        \draw [->] (2b) to node[above=-.7mm] {$a$} (3a) ;
        \draw [<-] (2a) to node[left=-.7mm] {$b$} (1a) ;
        \draw [->] (1b) to node[right=-.7mm] {$c$} (3b) ;
    
        \node[small, above left = 1cm and 5cm of 1] (4) {} ;
        \node[small,  above right = 1.5cm and 1cm of 4] (5) {} ;
        \node[zero, right = 0.3mm of 5] (5') {} ;
        \node[small,  below right = 1cm and 2cm of 4] (6) {} ;
        \node[zero, above = 0.3mm of 6] (6') {} ;
        \draw [->] (4) to [out=5, in=-95] node[above left = -.7mm and -.3mm] (c) {$u$}  (5);
    
        \draw [-] (5) to [out=-85, in=175] node[pos=0.457, small] (56) {}  (6);
        \draw [->] (5') to [out=-85, in=175] node[above right =-.7mm and -.7mm] {$v$} (6');
        \draw [->] (4) to node[below=-.7mm, pos=0.6] {$a$} (56) ;
        \draw [->] (4) to [out=-40, in=180]  (6);
     
  \end{tikzpicture}}}
  \caption{Case~\ref{case-end-close} (left) and~\ref{case-end-far} (right) from the proof of Theorem~\ref{thm-SLP-for-shortlex}.}
  \label{fig-cutting-off-peak} 
\end{figure}
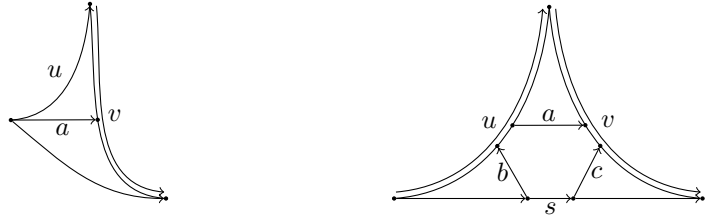

Theorem~\ref{thm-SLP-for-shortlex} now solves the compressed word problem. 

\begin{corollary} 
\label{coro-CWP-hyperbolic}
The compressed word problem for a hyperbolic group can be solved in polynomial time.
\end{corollary} 

\begin{proof}
Suppose that $\calG$ is the given compressed word.  
Note that $\eval(\calG) \in \Sigma^*$ represents $1_G$ if and only if $\slex(\eval(\calG)) = \emptyword$.  
This, in turn, happens if and only if $\slex(\eval(\calG))$ has length zero. 
By Theorem~\ref{thm-SLP-for-shortlex} we can compute in polynomial time a 
straight-line program $\calG'$ for $\slex(\eval(\calG))$ and by \cite[Proposition~3.9]{Loh14} we can compute the length of $\eval(\calG')$
in polynomial time. This concludes the proof. 
\end{proof}

\section{Further compressed decision problems}

\subsection{Compressed order problem} 
\label{sec-order}

Suppose that $G$ is a group. 
Suppose that $\Sigma$ is a finite, symmetric generating set for $G$. 
For any $g \in G$ we define the \emph{order} of $g$ to be the smallest positive integer $k$ so that $g^k = 1_G$.  
If there is no such $k$ we define the order to be infinity. 
We define the \emph{compressed order problem} as follows. 

\begin{description}
\item[Input] Straight-line program $\calG$ over $\Sigma$.
\item[Output] The order of the group element $\eval(\calG)$.
\end{description}

As a consequence of Corollary~\ref{coro-CWP-hyperbolic} we have the following result.

\begin{corollary} 
\label{coro-order}
Suppose that $G$ is a hyperbolic group.  
Then the compressed order problem for $G$ can be solved in polynomial time.
\end{corollary}

\begin{proof}
Suppose that $\calG$ is the given compressed word.
Suppose that $g \in G$ has finite order. 
Suppose that $\delta$ is the hyperbolicity constant for $G$. 
Then the order of $g$ is at most $2\delta + 1$;
see~\cite{BG95}.

To compute the order of $\eval(\calG)$ it suffices to check whether $\eval(\calG)^k =_G 1_G$, for some integer $k$ between $1$ and $2\delta + 1$ (inclusive).
Proposition~\ref{prop-power} gives us the desired compressed word and Corollary~\ref{coro-CWP-hyperbolic} checks it, both in polynomial time. 
\end{proof}

\subsection{The compressed (simultaneous) conjugacy and compressed centraliser problems} 
\label{sec-conjugacy}

Suppose that $G$ is a group. 
Suppose that $\Sigma$ is a finite, symmetric generating set for $G$. 
For group elements $g, h \in G$ we have the standard abbreviation $g^h = h^{-1} g h$.
If $\calL = (g_1, \ldots, g_k)$ is a finite list of group elements then we write $\calL^h = (g_1^h, \ldots, g_k^h)$.
We extend these definitions to words over $\Sigma$ in the obvious way.

\subsection{The problems} \label{the problems}

The \emph{compressed conjugacy problem} for $G$ is the following.

\begin{description}
\item[Input] 
Straight-line programs $\calG$ and $\calH$ over $\Sigma$.
\item[Question] 
Do $\eval(\calG)$ and $\eval(\calH)$ represent conjugate elements in $G$?  
\end{description}

If $\calL$ is a list of straight-line programs over $\Sigma$, then we define $\eval(\calL)$ to be the corresponding list of evaluations.  
We now define the \emph{compressed simultaneous conjugacy problem} for $G$.
\begin{description}
\item[Input] 
Finite lists $\calL = (\calG_1,\ldots,\calG_k)$ and $\calM = (\calH_1,\ldots,\calH_k)$ of straight-line programs over $\Sigma$.
\item[Question] 
Are $\eval(\calL)$ and $\eval(\calM)$ conjugate lists in $G$? 
\end{description}

In the case when the answer to either of these questions is positive, we might also want to compute a straight-line program for an element that conjugates $\eval(\calG)$ to $\eval(\calH)$ or $\eval(\calL)$ to $\eval(\calM)$.  

The \emph{compressed centraliser problem} for $G$ is the following computation problem.

\begin{description}
\item[Input] 
A finite list $\calL = (\calG_1, \ldots, \calG_k)$ of straight-line programs over $\Sigma$.
\item[Output] 
A finite list $\calM = (\calH_1, \ldots, \calH_l)$ such $\eval(\calM)$ generates the intersection of the centralisers of the elements $\eval(\calL)$. 
\end{description}

Note that this intersection is in fact the centraliser of the subgroup generated by the elements $\eval(\calL)$.  
When the desired centraliser is not finitely generated, by convention the problem has no solution. 

\subsection{The proofs}

A linear-time algorithm for solving the conjugacy problem in a hyperbolic group $G$ is described in~\cite[Section~3]{EpsteinH06}.  
For the  (uncompressed) simultaneous conjugacy problem, a quadratic time algorithm for torsion-free hyperbolic groups was presented in \cite{BrHo05}.
This was generalised in~\cite{BuckleyHolt13} to linear-time algorithms for the uncompressed simultaneous conjugacy, and the centraliser, problems in all 
hyperbolic groups.
We will show that essentially the same algorithms can be used to solve the compressed (simultaneous) conjugacy problem and the compressed centraliser problem, in polynomial time. 

We deal with the compressed conjugacy problem in Section~\ref{sec-conjugacy-1}.
Building on that, and making the special assumption that one of the input elements has infinite order, we solve the compressed simultaneous conjugacy problem and the compressed centraliser problem in Section~\ref{sec-infinite-order}.
Finally, we deal with the case that all input group elements have finite order in Section~\ref{sec-finite-order}.

\subsubsection{Compressed conjugacy problem} 
\label{sec-conjugacy-1}

We now have the following. 

\begin{theorem}
\label{thm-comp-conj}
Let $G$ be a hyperbolic group. 
Then the compressed conjugacy problem in $G$ is polynomial time. 
\end{theorem}

\begin{proof}
The input consists of two straight-line programs $\calG$ and $\calH$; we wish to test if $u = \eval(\calG)$ and
$v = \eval(\calH)$ are conjugate.  
To do this we essentially use the conjugacy algorithm from~\cite[Theorem~1.1]{EpsteinH06}, applied to the words $u$ and $v$. 
We will describe our modification of their algorithm, step-by-step, in the following.

Our description of each step consists of two parts.  
The first describes operations relating to the words $u$ and $v$; the second explains how we effect these operations in polynomial time using only the straight-line programs $\calG$ and $\calH$.  
All assertions that we make in the uncompressed setting are justified in~\cite{EpsteinH06}.
All corresponding assertions are then justified again, in the compressed setting, using the work in previous sections of this paper. 

Let $\delta$ be a positive integer that serves as a thinness constant for the Cayley graph $\Gamma = \Gamma(G, \Sigma)$; see Section~\ref{sec-hyp}.
We define constants $L = 34\delta + 2$ and $K = 17(2L + 1)/7$; see \cite[pages~298]{EpsteinH06}.

A word $w \in \Sigma^*$ is said to be \emph{shortlex straight} if, for all non-negative powers $k$, the word $w^k$ is shortlex reduced.
Applying Lemma~\ref{lem-shortlex-regular} and Proposition~\ref{prop-fsa-2} we can determine, in polynomial time, if a given compressed word $\eval(\calG)$ is shortlex straight.

In the preprocessing stage, we make a look-up table of all pairs of shortlex reduced words of length at most $K$ that are conjugate in $G$.

\begin{step}
\label{step-slex}
We replace $u$ and $v$ by $\slex(u)$ and $\slex(v)$.

By Theorem~\ref{thm-SLP-for-shortlex} we can replace, in polynomial time, the programs $\calG$ and $\calH$ by straight-line programs for $\slex(\eval(\calG))$ and $\slex(\eval(\calH))$, respectively. 
\end{step}

\begin{step}
\label{step-rotate}
For a word $w$, we define $w_C = w_R w_L$, where $w = w_L w_R$ with $|w_L| \leq |w_R| \leq |w_L| + 1$.
Replace $u$ by $\slex(u_C)$ and $v$ by $\slex(v_C)$.

Using cut operators and Theorem~\ref{thm-SLP-for-shortlex} we can make the corresponding substitutions on $\calG$ and $\calH$.
\end{step}

\begin{step}
\label{step-short}
If $|u|, |v| \leq K$ then use the look-up table to test for conjugacy of $u$ and $v$. 
Otherwise, at least one of the words, $u$ say, satisfies $|u| \geq K > 2L + 1$. 
If $|v| < 2L + 1$ then $u$ and $v$ are not conjugate~\cite[Section~3.1]{EpsteinH06}, and we return false. 
We assume from now on that $|u|, |v| \geq 2L + 1$. 

For the compressed conjugacy problem, if $|\eval(\calG)|, |\eval(\calH)| \leq K$ then we can compute $\eval(\calG)$ and $\eval(\calH)$ explicitly.
We then proceed as in the uncompressed setting.
\end{step}

\begin{step}
\label{step-straight}
There exists a group element $g \in \Ball(4\delta)$ and a positive integer $m$, of size at most $|\Ball(4 \delta)|^2$, such that the shortlex reduction of $g^{-1} u^m g$ is shortlex straight~\cite[Section~3.2]{EpsteinH06}. 
To find such, for every pair $(g, m)$ of at most those sizes, we replace $u$ by $\slex(g^{-1} u g)$ and test $z = \slex(u^m)$ to see if it is shortlex straight. 

Using Proposition~\ref{prop-fsa-2}, we can perform the corresponding operations with $\calG$.  Thus we find $g$ and $m$ and also find a straight-line program $\calG'$ with $\eval(\calG') = z$.
\end{step}

\begin{step}
\label{step-power}
We now test for the following necessary (but not sufficient) property for
the conjugacy of $u$ and $v$: is $v^m$ conjugate to $z$?  
We decide this as follows.  
For all $h \in \Ball(6\delta)$, 
we compute $v_h = \slex(h v^m h^{-1})$, and then test whether $v_h$ is a rotation of $z$.  
If this fails for all $h$, then $v^m$ and $z =_G u^m$ are not conjugate~\cite[Section~3.3]{EpsteinH06}.  
But then, $u$ and $v$ are not conjugate,  so we may stop and return false. 

Otherwise, we find $h$ and $v_h$ with this property.
Let $z_1$ be a prefix of $z$ such that $z =_G z_1 h v^m h^{-1} z_1^{-1}$.  
We replace $v$ by $\slex(z_1 h v h^{-1} z_1^{-1})$. 
Now we have $v^m =_G u^m =_G z$. 
From this we get that every $g \in G$ with $g^{-1} u g =_G v$ belongs to the centraliser $C_G(z)$ of $z$ in $G$.
In particular, $u$ and $v$ are conjugate in $G$ if and only if they are conjugate in  $C_G(z)$.

Using Corollary~\ref{coro-cyclic-conj} we can do the corresponding calculations with $\calH$ and $\calG'$. 
Checking whether $v_h$ is a rotation of $z$ can be accomplished in polynomial time by the first statement of Corollary~\ref{coro-cyclic-conj}; the second statement allows us to compute in polynomial time a straight-line program for 
$z_1$.
\end{step}

\begin{step}
\label{step-root}
Find the shortest prefix $y$ of $z$ that is a \emph{root} of $z$: that is, there is an $\ell \geq 1$ so that $z = y^\ell$.  
We do that by finding the second occurrence of the substring $z$ in the word $z^2$.

To find the root of $\eval(\calG')$, we compute a straight-line program for  $\eval(\calG')^2$ and appeal to Theorem~\ref{thm-pattern-matching}.  
We then build a straight-line program $\calG''$ with $\eval(\calG'') = y$ using cut operators and Theorem~\ref{thm-hagenah}.
\end{step}

\begin{step}
\label{step-candidates}
For each $h \in \Ball(2\delta)$, compute $\slex(h z h^{-1})$ and test whether it is a rotation of $z$.
If so, find a prefix $z_h$ of $z$ with $h z h^{-1} =_G z_h^{-1}zz_h$, and compute and store $\slex(z_h \cdot h)$ (which lies in $C_G(z)$) in a list $C_z$. 
Then $|C_z| \leq J = |\Ball(2\delta)|$.

Corollary~\ref{coro-cyclic-conj} allows us to do the corresponding calculations with $\calG'$.
We obtain a list of straight-line programs that evaluate to the words in the list $C_z$.
\end{step}

\begin{step}
\label{step-check}
For each $n$ with $0 \leq n \leq (J - 1)!$ and for each $z' \in C_z$, let $g = y^n z'$.
Test if $u =_G g v g^{-1}$. 
If so, then return true (and a conjugating element). 
If not, then return false because $u$ and $v$ are not conjugate~\cite[Section~3.4]{EpsteinH06}.

We can perform corresponding operations on the straight-line programs.
\end{step}

This concludes our description of a polynomial-time algorithm for the compressed conjugacy problem.  The correctness proof is identical to that in~\cite[Section~3]{EpsteinH06}
\end{proof}

\subsubsection{Compressed simultaneous conjugacy and centralisers: the infinite order case}
\label{sec-infinite-order}

We now turn to the following. 

\begin{theorem}
\label{thm-compcp}
Let $G$ be a hyperbolic group. 
Then the compressed simultaneous conjugacy problem for $G$ can be solved in polynomial time.
Moreover, if the two input lists are conjugate, then we can compute a straight-line program for a conjugating element in polynomial time. 
\end{theorem}

\begin{theorem}
\label{thm-comp-central}
Let $G$ be a hyperbolic group. 
Then the compressed centraliser problem for $G$ can be solved in polynomial time.
\end{theorem}

The input now consists of two lists $\calL = (\calG_1,\ldots,\calG_k)$ and $\calM = (\calH_1,\ldots,\calH_k)$ of straight-line programs over the alphabet $\Sigma$.
For the compressed centraliser problem we assume that $\calL = \calM$.
For all $i$ we let $u_i = \eval(\calG_i)$ and $v_i = \eval(\calH_i)$. 

By Corollary~\ref{coro-order} we can check in polynomial time whether some $u_i$ has infinite order.
Following~\cite[Section~3]{BuckleyHolt13} we begin by assuming that this is indeed the case. 
Reordering the lists in the same way, as needed, we may assume that $u_1$ has infinite order. 
If $v_1$ does not have infinite order we are done. 

The conjugacy testing algorithm proceeds as follows.  We first repeatedly replace the elements in $\calL$ by conjugates, using a common conjugating element.  
This culminates in a check for a conjugating element which must lie in an explicit finite set. 
In each replacement, straight-line programs are known for the conjugating element.  
By keeping track of these, we can find (if the lists are conjugate) an overall conjugating element for the original input.   
We omit further details regarding this overall conjugating element.

We proceed by carrying out the eight steps of the algorithm of Section~\ref{sec-conjugacy-1} as applied to $u_1$ and $v_1$.  
If they are not conjugate we are done. 
Suppose that they are conjugate.  
In this case we record the programs for the words $z$ and $y$ produced by Steps~\ref{step-power} and~~\ref{step-root}.  We also record the list (of straight-line programs) $C_z$ given in Step~\ref{step-candidates}.
The overall algorithm also gives us a straight-line program for an element $g \in G$ with $u_1^g =_G v_1$.
The algorithm also replaced $u_1$ and $v_1$ by conjugates in some of the steps; we make the corresponding replacements to the other elements of $\calL$ and $\calM$. 
By replacing each $u_i$ by its conjugate under $g$, we may now assume that $u_1 = v_1$.

Thus we have reduced the problem to the following.  
Assuming that $u_1 = v_1$ and that $u_1$ has infinite order, we must decide if there is $g \in C_G(u_1)$ with $u_i^g =_G v_i$ for $2 \leq i \leq k$.   We are also given $z$; thus $u_1^m =_G v_1^m =_G z$ and $z$ is shortlex straight element $z$ and $m \geq 1$. 
We are also given $y$ with $z = y^\ell$ and for maximal $\ell \geq 1$. 

In~\cite[Section~3.4]{EpsteinH06} it is shown that all elements $g \in C_G(z)$ have the form $g =_G y^n z'$, for some $n \in \ZZ$ and $z' \in C_z$, where $C_z$ is the list given above. 
So the same applies to any $g \in C_G(u_1) \subseteq C_G(z)$.

We now try each $z' \in C_z$ in turn.  Replacing each $v_i$ by $z' v_i (z')^{-1}$, the problem reduces to the following:
is there some $n \in \ZZ$ such that $u_i^{y^n} =_G v_i$ for $1 \leq i \leq k$?

To solve this problem, we apply~\cite[Proposition~24]{BuckleyHolt13} to each pair $u_i, v_i$ in turn. 
For each $i$, there are three possibilities.
\begin{enumerate}
\item[(i)] 
there exist $0 \leq r_i < t_i \leq |\Ball(2\delta)|$ such that $u_i^{y^j} =_G v_i$ if and only if $j \equiv r_i \bmod t_i$;
\item[(ii)] 
there is a unique $r_i \in \ZZ$ with $u_i^{y^{r_i}} =_G v_i$, where $|r_i|$ is bounded by a linear function of $|u_i|$ and $|v_i|$; or
\item[(iii)] 
there is no $r_i \in \ZZ$ with $u_i^{y^{r_i}} =_G v_i$.
\end{enumerate}
The proof of~\cite[Proposition~24]{BuckleyHolt13} provides an algorithm for determining which case applies, and for finding $r_i, t_i$ in cases (i) and (ii).
This involves calculating a number of powers $u_i^n$, $v_i^n$, and $y^n$ for integers $n$ such that $|n|$ is bounded by a linear function of $|u_i|$ and $|v_i|$, and where the number of powers that need to be calculated is bounded by a constant. 
So we can perform these calculations in polynomial time with straight-line programs by Proposition~\ref{prop-power}.

After performing this calculation for each $i$ with $1 \leq i \leq k$, the conjugacy problem for the lists reduces to solving some modular linear equations involving the integers $r_i$ and $t_i$, as described in~\cite[Section~3.4]{BuckleyHolt13}. 
Since the $r_i$ and $t_i$ in case (i) are bounded  by a constant and, for $r_i$ in case (ii), $\log |r_i|$ is bounded by a linear function of the size of the straight-line programs representing $u_i$ and $v_i$, these equations can be solved in polynomial time using standard arithmetical operations on the binary representations of $r_i$ and $t_i$.  This completes our discussion of the compressed simultaneous conjugacy problem in the case where there is a list element of infinite order. 

For the compressed centraliser problem, we are in the same situation but with $v_i = u_i$ for all $i$.
We perform the same calculations as above, but we do them for every $z' \in C_z$.  
If there are solutions, then we find them by solving modular equations.  
The set of solutions we find now generates the centraliser.  
This completes our discussion of the compressed centraliser problem in the case where there is a list element of infinite order.

\subsubsection{Compressed simultaneous conjugacy and centralisers: the finite order case}
\label{sec-finite-order}

Here we continue the proofs of Theorems~\ref{thm-compcp} and~\ref{thm-comp-central}.  
We now consider the case where all of the $u_i$ (in the list $\eval(\calL)$) have finite order.  
We now follow~\cite[Section~4]{BuckleyHolt13}. 
No new complications arise when applying those methods to lists of straight-line programs. 
Indeed, some steps become easier because we are only interested in achieving polynomial, rather than linear, time.

We follow the steps of the algorithm described in~\cite[Section~4.5]{BuckleyHolt13}. 
We deal with the conjugacy and centraliser problems together; the two lists are taken to be equal for the centraliser calculation. 
At this stage we have already verified that all of the $u_i$ and $v_i$ have finite order.  
Furthermore all of the words are shortlex reduced. 
By deleting programs, we can assume that the list $u_1, \ldots, u_k$, and likewise the list $v_1, \ldots, v_k$, has no duplicates.
Thus the $u_i$ represent distinct group elements, as do the $v_i$.

Let $n = \min\{ |\Ball(2\delta)|^4 + 1, k \}$.
We consider the prefix sublists $\eval(\calL') = (u_1, \ldots, u_n)$ and $\eval(\calM') = (v_1, \ldots, v_n)$.
We apply the function \textsc{ShortenWords} from~\cite[Section~4.2]{BuckleyHolt13} to the lists $\calL'$ and $\calM'$. 
This function applies $\slex$ to a number of words; this number is bounded above by $n^2$.  
Each word is a concatenation (of length at most $n + 2$) of words either from the lists $\calL'$ or $\calM'$, or of words previously calculated during this process. 
These operations can be executed in polynomial time when working with straight-line programs.
Since there is an absolute bound $|\Ball(2\delta)|^4 + 1$ on the lengths of $\calL'$ and $\calM'$, the complete application of \textsc{ShortenWords} to each list takes place in polynomial time. 

\textsc{ShortenWords} has two possible outcomes.  
In the first it finds a product $u_r \cdot u_{r+1} \cdots u_s$ of elements of $\calL'$ with infinite order.  
This reduces the problem to the case dealt with in Section~\ref{sec-infinite-order}.  

In the second possible outcome, \textsc{ShortenWords} replaces $\calL'$ and $\calM'$ by conjugates and then calculates lists $\calL'' = (u_1', \ldots, u_n')$ and $\calM'' = (v_1', \ldots, v_n')$ with $|u_i'|$ and $|v_i'|$ bounded by a constant, and such that $\calL^g = \calM$ if and only if $(\calL'')^g = \calM''$. 

We now test in time $O(1)$ (using our precomputed look-up table) whether there exists $g \in G$ with
\[
(u_1', \ldots, u_n')^g = (v_1', \ldots, v_n')
\]
If so, we replace $(u_1,\ldots,u_k)$ by $(u_1,\ldots,u_k)^g$ and thereby assume that $u_i=v_i$ for $1 \leq i \leq n$.
For the centraliser problem, methods are described in~\cite[Proposition~2.3]{GerstenShort91} of finding a generating set of the
centraliser of any quasiconvex subgroup of any biautomatic group; finitely generated subgroups of hyperbolic groups satisfy these
conditions.
Since they need only be applied to words of bounded length their complexity does not matter - indeed, we could precompute all such centralisers.

This completes the proof in the case $n = k$. 
In the case $k > n$, it is proved in~\cite[Corollary~30]{BuckleyHolt13} that the centraliser $C$ of the subgroup $\subgp{u_1, \ldots, u_n}$ is finite, and that the elements of $C$ have lengths bounded by a constant. 
So we can compute the elements of $C$ explicitly (in time $O(1)$).
Then we simply need to check whether any $g \in C$ satisfies
\[
(u_{n + 1}, \ldots, u_k)^g = (v_{n + 1}, \ldots, v_k)
\]
This completes the proofs of Theorems~\ref{thm-compcp} and~\ref{thm-comp-central}. \qed

\subsection{Compressed knapsack}

In this final section, we prove the following.

\begin{theorem} 
\label{thm-knapsack-hyp}
If $G$ is an infinite hyperbolic group then the compressed knapsack problem for $G$ is $\NP$--complete.
\end{theorem}

As above, fix $G$ a finitely generated group.  Fix as well a finite symmetric generating set $\Sigma$.
A \emph{knapsack expression} over $\Sigma$ is a regular expression of the form $E = u^{-1} u_1^* u_2^* \cdots u_k^*$ with $k \geq 0$ and $u,u_i \in \Sigma^*$. 
The \emph{length} of $E$ is defined to be  $|E| = |u| + \sum_{i = 1}^k |u_i|$.
A \emph{solution} for $E$ is a tuple $(n_1, n_2, \ldots, n_k) \in \NN^k$ of natural numbers such that $u =_G u_1^{n_1} u_2^{n_2} \cdots u_k^{n_k}$.  In other words:
the language defined by $E$ contains a word that represents the identity of $G$.

The \emph{knapsack problem} for $G$, over $\Sigma$, is the following.
\begin{description}
\item[Input] A knapsack expression $E$ over $\Sigma$.
\item[Question] Does $E$ has a solution?
\end{description}
In \cite[Theorem~6.1]{MyNiUs14} it was shown that the knapsack problem for a hyperbolic group can be solved in polynomial time.
A crucial step in the proof for this fact is the following result, which is of independent interest.
 
\begin{theorem}[\mbox{\cite[Theorem~6.7]{MyNiUs14}}] 
\label{thm-ushakov}
For every hyperbolic group $G$ there exists a polynomial $p(x)$ such that the following holds.
Suppose that a knapsack expression $E = u^{-1} u_1^* u_2^*\cdots u_{k}^*$ over $G$ has a solution.  
Then $E$ has a solution $(n_1, n_2, \ldots, n_k) \in \NN^k$ such that $n_i \leq p(|E|)$ for all $i$ satisfying $1 \leq i \leq k$. \qed
\end{theorem}    

Recently, this result has been extended to acylindrically hyperbolic group in \cite{BierBog21}.

Let us now consider the \emph{compressed knapsack problem} for $G$. 
It is defined in the same way as the knapsack problem, except that the words $u,u_i \in \Sigma^*$ are given by straight-line programs.
Note that the compressed knapsack problem for $\ZZ$ is $\NP$--complete~\cite[Proposition 4.1.1]{Haa11}. 
Hence, for every group with an element of infinite order, the compressed knapsack problem is $\NP$--hard.
This makes it interesting to look for groups where the compressed knapsack problem is $\NP$--complete.

From Corollary~\ref{coro-CWP-hyperbolic} and Theorem~\ref{thm-ushakov} we prove Theorem~\ref{thm-knapsack-hyp},
which states that compressed knapsack for an infinite hyperbolic group $G$ is $\NP$--complete.

\begin{proof}[Proof of Theorem~\ref{thm-knapsack-hyp}]
Consider a knapsack expression $E = u^{-1} u_1^* u_2^* \cdots u_{k}^*$ over $G$, where $u$ and the $u_i$ are given by straight-line programs $\calG$ and $\calG_i$.
We then have $|u|,|u_i| \leq 3^{|\calG_i|/3}$ by Lemma~\ref{lem-SLP-upper-bound}.
Let $N = |\calG| + \sum_{i = 1}^k |\calG_i|$ be the input length.

By Theorem~\ref{thm-ushakov}, there exists a polynomial $p(x)$ such that $E$ has a solution if and only if it has a solution $(n_1, n_2 \ldots, n_k) \in \NN^k$ with $n_i \leq p(|E|)$ for all $i$ so that $1 \leq i \leq k$.
Thus we obtain a bound of the form $2^{O(N)}$ on the exponents $n_i$. 
Hence, we can guess the binary encoding of a tuple $(n_1, n_2, \ldots, n_k) \in \NN^k$ with all $n_i$ bounded by $2^{O(N)}$ and then check whether it is a solution for $E$. 
The latter can be done in polynomial time by constructing from the straight-line programs $\calG$ and $\calG_i$ a straight-line program $\calH$ for $u^{-1} u_1^{n_1} u_2^{n_2} \cdots u_{k}^{n_{k}}$ using Proposition~\ref{prop-power}. 
Finally, we check in polynomial time whether $\eval(\calH) =_G 1$ using Corollary~\ref{coro-CWP-hyperbolic}. 

The $\NP$--hardness of the compressed knapsack problem for $G$ (an infinite hyperbolic group) now follows from the fact that $G$ has elements of infinite order~\cite[page~156]{GhysdelaHarpe90} and the above mentioned result for $\ZZ$~\cite[Proposition 4.1.1]{Haa11}.  
\end{proof}


\begin{thebibliography}{10}

\bibitem{Agol13}
Ian Agol.
\newblock The virtual {H}aken conjecture.
\newblock {\em Doc. Math.}, 18:1045--1087, 2013.
\newblock With an appendix by Ian Agol, Daniel Groves, and Jason Manning.
\newblock \href {http://arxiv.org/abs/1204.2810} {\path{arXiv:1204.2810}},
  \href {https://doi.org/10.1016/j.procs.2013.05.269}
  {\path{doi:10.1016/j.procs.2013.05.269}}.

\bibitem{ABCFLMSS91}
J.~M. Alonso, T.~Brady, D.~Cooper, V.~Ferlini, M.~Lustig, M.~Mihalik,
  M.~Shapiro, and H.~Short.
\newblock Notes on word hyperbolic groups.
\newblock In {\em Group theory from a geometrical viewpoint ({T}rieste, 1990)},
  pages 3--63. World Sci. Publ., River Edge, NJ, 1991.
\newblock Edited by Short.
\newblock \href {https://doi.org/10.1142/1235} {\path{doi:10.1142/1235}}.

\bibitem{BabSzem84}
L{\'{a}}szl{\'{o}} Babai and Endre Szemer{\'e}di.
\newblock On the complexity of matrix group problems {I}.
\newblock In {\em Proceedings of the 25th Annual Symposium on Foundations of
  Computer Science, FOCS 1984}, pages 229--240. IEEE Computer Society, 1984.
\newblock \href {https://doi.org/10.1109/SFCS.1984.715919}
  {\path{doi:10.1109/SFCS.1984.715919}}.

\bibitem{BartholdiFLW19}
Laurent Bartholdi, Michael Figelius, Markus Lohrey, and Armin Wei{\ss}.
\newblock Groups with {ALOGTIME}-hard word problems and {PSPACE}-complete
  compressed word problems.
\newblock {\em ACM Transactions on Computation Theory}, 14(3–4), 2023.
\newblock \href {https://doi.org/10.1145/3569708} {\path{doi:10.1145/3569708}}.

\bibitem{Baumslag69}
Gilbert Baumslag.
\newblock A non-cyclic one-relator group all of whose finite quotients are
  cyclic.
\newblock {\em J. Austral. Math. Soc.}, 10:497--498, 1969.
\newblock \href {https://doi.org/10.1017/S1446788700007783}
  {\path{doi:10.1017/S1446788700007783}}.

\bibitem{BaumslagSolitar62}
Gilbert Baumslag and Donald Solitar.
\newblock Some two-generator one-relator non-{H}opfian groups.
\newblock {\em Bull. Amer. Math. Soc.}, 68:199--201, 1962.
\newblock \href {https://doi.org/10.1090/S0002-9904-1962-10745-9}
  {\path{doi:10.1090/S0002-9904-1962-10745-9}}.

\bibitem{BeMcPeTh97}
Martin Beaudry, Pierre McKenzie, Pierre P\'{e}ladeau, and Denis Th\'{e}rien.
\newblock Finite monoids: from word to circuit evaluation.
\newblock {\em SIAM J. Comput.}, 26(1):138--152, 1997.
\newblock \href {https://doi.org/10.1137/S0097539793249530}
  {\path{doi:10.1137/S0097539793249530}}.

\bibitem{BierBog21}
Agnieszka Bier and Oleg Bogopolski.
\newblock Exponential equations in acylindrically hyperbolic groups, 2021.
\newblock \href {http://arxiv.org/abs/2106.11385} {\path{arXiv:2106.11385}}.

\bibitem{BG95}
O.~V. Bogopol\cprime~ski\u{\i} and V.~N. Gerasimov.
\newblock Finite subgroups of hyperbolic groups.
\newblock {\em Algebra i Logika}, 34(6):619--622, 728, 1995.
\newblock \href {https://doi.org/10.1007/BF00739330}
  {\path{doi:10.1007/BF00739330}}.

\bibitem{Bowditch95}
Brian~H. Bowditch.
\newblock A short proof that a subquadratic isoperimetric inequality implies a
  linear one.
\newblock {\em Michigan Math. J.}, 42(1):103--107, 1995.
\newblock \href {https://doi.org/10.1307/mmj/1029005156}
  {\path{doi:10.1307/mmj/1029005156}}.

\bibitem{BrHo05}
Martin~R. Bridson and James Howie.
\newblock Conjugacy of finite subsets in hyperbolic groups.
\newblock {\em Internat. J. Algebra Comput.}, 15(4):725--756, 2005.
\newblock \href {https://doi.org/10.1142/S0218196705002529}
  {\path{doi:10.1142/S0218196705002529}}.

\bibitem{BuckleyHolt13}
David~J. Buckley and Derek~F. Holt.
\newblock The conjugacy problem in hyperbolic groups for finite lists of group
  elements.
\newblock {\em Internat. J. Algebra Comput.}, 23(5):1127--1150, 2013.
\newblock \href {http://arxiv.org/abs/1111.1554} {\path{arXiv:1111.1554}},
  \href {https://doi.org/10.1142/S0218196713500203}
  {\path{doi:10.1142/S0218196713500203}}.

\bibitem{ChandlerMagnus82}
Bruce Chandler and Wilhelm Magnus.
\newblock {\em The history of combinatorial group theory}, volume~9 of {\em
  Studies in the History of Mathematics and Physical Sciences}.
\newblock Springer-Verlag, New York, 1982.
\newblock A case study in the history of ideas.
\newblock \href {https://doi.org/10.1007/978-1-4613-9487-7}
  {\path{doi:10.1007/978-1-4613-9487-7}}.

\bibitem{CLLLPPSS05}
Moses Charikar, Eric Lehman, Ding Liu, Rina Panigrahy, Manoj Prabhakaran, Amit
  Sahai, and abhi shelat.
\newblock The smallest grammar problem.
\newblock {\em IEEE Trans. Inform. Theory}, 51(7):2554--2576, 2005.
\newblock \href {https://doi.org/10.1109/TIT.2005.850116}
  {\path{doi:10.1109/TIT.2005.850116}}.

\bibitem{CiobanuElder19}
Laura Ciobanu and Murray Elder.
\newblock The complexity of solution sets to equations in hyperbolic groups.
\newblock {\em Israel Journal of Mathematics}, 245:869–920, 2021.
\newblock \href {https://doi.org/10.1007/s11856-021-2232-z}
  {\path{doi:10.1007/s11856-021-2232-z}}.

\bibitem{DaGui10}
Fran\c{c}ois Dahmani and Vincent Guirardel.
\newblock Foliations for solving equations in groups: free, virtually free, and
  hyperbolic groups.
\newblock {\em J. Topol.}, 3(2):343--404, 2010.
\newblock \href {http://arxiv.org/abs/1902.07349} {\path{arXiv:1902.07349}},
  \href {https://doi.org/10.1112/jtopol/jtq010}
  {\path{doi:10.1112/jtopol/jtq010}}.

\bibitem{DaGui11}
Fran\c{c}ois Dahmani and Vincent Guirardel.
\newblock The isomorphism problem for all hyperbolic groups.
\newblock {\em Geom. Funct. Anal.}, 21(2):223--300, 2011.
\newblock \href {http://arxiv.org/abs/1002.2590} {\path{arXiv:1002.2590}},
  \href {https://doi.org/10.1007/s00039-011-0120-0}
  {\path{doi:10.1007/s00039-011-0120-0}}.

\bibitem{Dehn11}
Max Dehn.
\newblock \"{U}ber unendliche diskontinuierliche {G}ruppen.
\newblock {\em Math. Ann.}, 71(1):116--144, 1911.
\newblock \href {https://doi.org/10.1007/BF01456932}
  {\path{doi:10.1007/BF01456932}}.

\bibitem{DiekertKM13}
Volker Diekert, Olga Kharlampovich, and Atefeh~Mohajeri Moghaddam.
\newblock S{LP} compression for solutions of equations with constraints in free
  and hyperbolic groups.
\newblock {\em Internat. J. Algebra Comput.}, 25(1-2):81--111, 2015.
\newblock \href {http://arxiv.org/abs/1308.5586} {\path{arXiv:1308.5586}},
  \href {https://doi.org/10.1142/S0218196715400056}
  {\path{doi:10.1142/S0218196715400056}}.

\bibitem{DiekertLU12}
Volker Diekert, J\"{u}rn Laun, and Alexander Ushakov.
\newblock Efficient algorithms for highly compressed data: the word problem in
  {H}igman's group is in {P}.
\newblock {\em Internat. J. Algebra Comput.}, 22(8):1240008, 19, 2012.
\newblock \href {http://arxiv.org/abs/1103.1232} {\path{arXiv:1103.1232}},
  \href {https://doi.org/10.1142/S0218196712400085}
  {\path{doi:10.1142/S0218196712400085}}.

\bibitem{DisonER16}
Will Dison, Eduard Einstein, and Timothy~R. Riley.
\newblock Ackermannian integer compression and the word problem for hydra
  groups.
\newblock In {\em 41st {I}nternational {S}ymposium on {M}athematical
  {F}oundations of {C}omputer {S}cience}, volume~58 of {\em LIPIcs. Leibniz
  Int. Proc. Inform.}, pages Art. No. 30, 14. Schloss Dagstuhl. Leibniz-Zent.
  Inform., Wadern, 2016.
\newblock \href {https://doi.org/10.4230/LIPIcs.MFCS.2016.30}
  {\path{doi:10.4230/LIPIcs.MFCS.2016.30}}.

\bibitem{DisonER18}
Will Dison, Eduard Einstein, and Timothy~R. Riley.
\newblock Taming the hydra: the word problem and extreme integer compression.
\newblock {\em Internat. J. Algebra Comput.}, 28(7):1299--1381, 2018.
\newblock \href {https://doi.org/10.1142/S0218196718500583}
  {\path{doi:10.1142/S0218196718500583}}.

\bibitem{DisonRiley13}
Will Dison and Timothy~R. Riley.
\newblock Hydra groups.
\newblock {\em Comment. Math. Helv.}, 88(3):507--540, 2013.
\newblock \href {https://doi.org/10.4171/CMH/294} {\path{doi:10.4171/CMH/294}}.

\bibitem{ECHLPT}
David B.~A. Epstein, James~W. Cannon, Derek~F. Holt, Silvio V.~F. Levy,
  Michael~S. Paterson, and William~P. Thurston.
\newblock {\em Word processing in groups}.
\newblock Jones and Bartlett Publishers, Boston, MA, 1992.
\newblock \href {https://doi.org/10.1201/9781439865699}
  {\path{doi:10.1201/9781439865699}}.

\bibitem{EpsteinH06}
David B.~A. Epstein and Derek~F. Holt.
\newblock The linearity of the conjugacy problem in word-hyperbolic groups.
\newblock {\em Internat. J. Algebra Comput.}, 16(2):287--305, 2006.
\newblock \href {https://doi.org/10.1142/S0218196706002986}
  {\path{doi:10.1142/S0218196706002986}}.

\bibitem{FrenkelNU15}
Elizaveta Frenkel, Andrey Nikolaev, and Alexander Ushakov.
\newblock Knapsack problems in products of groups.
\newblock {\em J. Symbolic Comput.}, 74:96--108, 2016.
\newblock \href {https://doi.org/10.1016/j.jsc.2015.05.006}
  {\path{doi:10.1016/j.jsc.2015.05.006}}.

\bibitem{GanardiKLZ18}
Moses Ganardi, Daniel K\"{o}nig, Markus Lohrey, and Georg Zetzsche.
\newblock Knapsack problems for wreath products.
\newblock In {\em 35th {S}ymposium on {T}heoretical {A}spects of {C}omputer
  {S}cience}, volume~96 of {\em LIPIcs. Leibniz Int. Proc. Inform.}, pages Art.
  No. 32, 13. Schloss Dagstuhl. Leibniz-Zent. Inform., Wadern, 2018.
\newblock \href {https://doi.org/10.4230/LIPIcs.STACS.2018.32}
  {\path{doi:10.4230/LIPIcs.STACS.2018.32}}.

\bibitem{GaZa91}
Max Garzon and Yechezkel Zalcstein.
\newblock The complexity of {G}rigorchuk groups with application to
  cryptography.
\newblock {\em Theoretical Computer Science}, 88(1):83--98, 1991.
\newblock \href {https://doi.org/10.1016/0304-3975(91)90074-C}
  {\path{doi:10.1016/0304-3975(91)90074-C}}.

\bibitem{GerstenShort91}
Stephen~M. Gersten and Hamish~B. Short.
\newblock Rational subgroups of biautomatic groups.
\newblock {\em Ann. of Math. (2)}, 134(1):125--158, 1991.
\newblock \href {https://doi.org/10.2307/2944334} {\path{doi:10.2307/2944334}}.

\bibitem{GhysdelaHarpe90}
\'{E}tienne Ghys and Pierre de~la Harpe.
\newblock Panorama.
\newblock In {\em Sur les groupes hyperboliques d'apr\`es {M}ikhael {G}romov
  ({B}ern, 1988)}, volume~83 of {\em Progr. Math.}, pages 1--25. Birkh\"{a}user
  Boston, Boston, MA, 1990.
\newblock \href {https://doi.org/10.1007/978-1-4684-9167-8_1}
  {\path{doi:10.1007/978-1-4684-9167-8_1}}.

\bibitem{Gro87}
Mikhail~L. Gromov.
\newblock Hyperbolic groups.
\newblock In {\em Essays in group theory}, volume~8 of {\em Math. Sci. Res.
  Inst. Publ.}, pages 75--263. Springer, New York, 1987.
\newblock \href {https://doi.org/10.1007/978-1-4613-9586-7_3}
  {\path{doi:10.1007/978-1-4613-9586-7_3}}.

\bibitem{Haa11}
Christoph Haase.
\newblock {\em On the complexity of model checking counter automata}.
\newblock PhD thesis, University of {Oxford}, St Catherine's College, 2011.

\bibitem{Hag00}
Christian Hagenah.
\newblock {\em Gleichungen mit regul{\"a}ren Randbedingungen {\"u}ber freien
  Gruppen}.
\newblock PhD thesis, University of {Stuttgart}, 2000.

\bibitem{HagWi10}
Fr\'{e}d\'{e}ric Haglund and Daniel~T. Wise.
\newblock Coxeter groups are virtually special.
\newblock {\em Adv. Math.}, 224(5):1890--1903, 2010.
\newblock \href {https://doi.org/10.1016/j.aim.2010.01.011}
  {\path{doi:10.1016/j.aim.2010.01.011}}.

\bibitem{HauLohHau13}
Nico Haubold, Markus Lohrey, and Christian Mathissen.
\newblock Compressed decision problems for graph products of groups and
  applications to (outer) automorphism groups.
\newblock {\em International Journal of Algebra and Computation}, 22(8), 2013.
\newblock \href {https://doi.org/10.1142/S0218196712400073}
  {\path{doi:10.1142/S0218196712400073}}.

\bibitem{Higman51}
Graham Higman.
\newblock A finitely generated infinite simple group.
\newblock {\em J. London Math. Soc.}, 26:61--64, 1951.
\newblock \href {https://doi.org/10.1112/jlms/s1-26.1.61}
  {\path{doi:10.1112/jlms/s1-26.1.61}}.

\bibitem{HirshfeldJM94}
Yoram Hirshfeld, Mark Jerrum, and Faron Moller.
\newblock A polynomial-time algorithm for deciding equivalence of normed
  context-free processes.
\newblock In {\em Proceedings of the 35th Annual Symposium on Foundations of
  Computer Science, FOCS 1994}, pages 623--631. IEEE Computer Society, 1994.
\newblock \href {https://doi.org/10.1109/SFCS.1994.365729}
  {\path{doi:10.1109/SFCS.1994.365729}}.

\bibitem{HiJeMo96}
Yoram Hirshfeld, Mark Jerrum, and Faron Moller.
\newblock A polynomial algorithm for deciding bisimilarity of normed
  context-free processes.
\newblock {\em Theoretical Computer Science}, 158(1\&2):143--159, 1996.
\newblock \href {https://doi.org/10.1016/0304-3975(95)00064-X}
  {\path{doi:10.1016/0304-3975(95)00064-X}}.

\bibitem{HoltLS19}
Derek~F. Holt, Markus Lohrey, and Saul Schleimer.
\newblock Compressed decision problems in hyperbolic groups.
\newblock In {\em Proceedings of the 36th International Symposium on
  Theoretical Aspects of Computer Science, {STACS} 2019}, volume 126 of {\em
  LIPIcs}, pages 37:1--37:16. Schloss Dagstuhl - Leibniz-Zentrum f{\"u}r
  Informatik, 2019.
\newblock \href {https://doi.org/10.4230/LIPIcs.STACS.2019.37}
  {\path{doi:10.4230/LIPIcs.STACS.2019.37}}.

\bibitem{HoltRees20}
Derek~F. Holt and Sarah Rees.
\newblock The compressed word problem in relatively hyperbolic groups.
\newblock {\em Journal of Algebra}, 607:305--343, 2022.
\newblock \href {https://doi.org/10.1016/j.jalgebra.2022.01.001}
  {\path{doi:10.1016/j.jalgebra.2022.01.001}}.

\bibitem{HoUl79}
John~E. Hopcroft and Jeffrey~D. Ullman.
\newblock {\em {Introduction to Automata Theory, Languages and Computation}}.
\newblock Addison--Wesley, Reading, MA, 1979.

\bibitem{Jez15}
Artur Je{\.{z}}.
\newblock Faster fully compressed pattern matching by recompression.
\newblock {\em {ACM} Transactions on Algorithms}, 11(3):20:1--20:43, 2015.
\newblock \href {https://doi.org/10.1145/2631920} {\path{doi:10.1145/2631920}}.

\bibitem{Karp72}
Richard~M. Karp.
\newblock Reducibility among combinatorial problems.
\newblock In R.~E. Miller and J.~W. Thatcher, editors, {\em Complexity of
  Computer Computations}, pages 85--103. Plenum Press, 1972.
\newblock \href {https://doi.org/10.1007/978-1-4684-2001-2\_9}
  {\path{doi:10.1007/978-1-4684-2001-2\_9}}.

\bibitem{KarpinskiEtAl95}
Marek Karpinski, Wojciech Rytter, and Ayumi Shinohara.
\newblock Pattern-matching for strings with short descriptions.
\newblock In {\em Combinatorial pattern matching (Espoo, 1995)}, volume 937 of
  {\em Lecture Notes in Comput. Sci.}, pages 205--214. Springer, Berlin, 1995.
\newblock \href {https://doi.org/10.1007/3-540-60044-2\_44}
  {\path{doi:10.1007/3-540-60044-2\_44}}.

\bibitem{KaMa12}
Martin Kassabov and Francesco Matucci.
\newblock The simultaneous conjugacy problem in groups of piecewise linear
  functions.
\newblock {\em Groups, Geometry, and Dynamics}, 6(2):279--315, 2012.
\newblock \href {https://doi.org/10.4171/GGD/158} {\path{doi:10.4171/GGD/158}}.

\bibitem{KonigL15}
Daniel K{\"{o}}nig and Markus Lohrey.
\newblock Evaluation of circuits over nilpotent and polycyclic groups.
\newblock {\em Algorithmica}, 80(5):1459--1492, 2018.
\newblock \href {https://doi.org/10.1007/s00453-017-0343-z}
  {\path{doi:10.1007/s00453-017-0343-z}}.

\bibitem{KoenigLohreyZetzsche2015a}
Daniel K{\"o}nig, Markus Lohrey, and Georg Zetzsche.
\newblock Knapsack and subset sum problems in nilpotent, polycyclic, and
  co-context-free groups.
\newblock In {\em Algebra and Computer Science}, volume 677 of {\em
  Contemporary Mathematics}, pages 138--153. American Mathematical Society,
  2016.
\newblock \href {https://doi.org/10.1090/conm/677}
  {\path{doi:10.1090/conm/677}}.

\bibitem{Loh06siam}
Markus Lohrey.
\newblock Word problems and membership problems on compressed words.
\newblock {\em SIAM Journal on Computing}, 35(5):1210 -- 1240, 2006.
\newblock \href {https://doi.org/10.1137/S0097539704445950}
  {\path{doi:10.1137/S0097539704445950}}.

\bibitem{Loh12survey}
Markus Lohrey.
\newblock Algorithmics on {SLP}-compressed strings: A survey.
\newblock {\em Groups Complexity Cryptology}, 4(2):241--299, 2012.
\newblock \href {https://doi.org/10.1515/gcc-2012-0016}
  {\path{doi:10.1515/gcc-2012-0016}}.

\bibitem{Loh14}
Markus Lohrey.
\newblock {\em The Compressed Word Problem for Groups}.
\newblock SpringerBriefs in Mathematics. Springer, 2014.
\newblock \href {https://doi.org/10.1007/978-1-4939-0748-9}
  {\path{doi:10.1007/978-1-4939-0748-9}}.

\bibitem{Loh2019}
Markus Lohrey.
\newblock Knapsack in hyperbolic groups.
\newblock {\em Journal of Algebra}, 545:390--415, 2020.
\newblock \href {https://doi.org/10.1016/j.jalgebra.2019.04.008}
  {\path{doi:10.1016/j.jalgebra.2019.04.008}}.

\bibitem{LohreyZ18}
Markus Lohrey and Georg Zetzsche.
\newblock Knapsack in graph groups.
\newblock {\em Theory of Computing Systems}, 62(1):192--246, 2018.
\newblock \href {https://doi.org/10.1007/s00224-017-9808-3}
  {\path{doi:10.1007/s00224-017-9808-3}}.

\bibitem{Macd09}
Jeremy Macdonald.
\newblock Compressed words and automorphisms in fully residually free groups.
\newblock {\em International Journal of Algebra and Computation},
  20(3):343--355, 2010.
\newblock \href {https://doi.org/10.1142/S021819671000542X}
  {\path{doi:10.1142/S021819671000542X}}.

\bibitem{MacDonaldMO17}
Jeremy MacDonald, Alexei~G. Myasnikov, and Denis Ovchinnikov.
\newblock Low-complexity computations for nilpotent subgroup problems.
\newblock {\em International Journal of Algebra and Computation},
  29(4):639--661, 2019.
\newblock \href {https://doi.org/10.1142/S021819671950019X}
  {\path{doi:10.1142/S021819671950019X}}.

\bibitem{MKS76}
Wilhelm Magnus, Abraham Karrass, and Donald Solitar.
\newblock {\em Combinatorial group theory}.
\newblock Dover Publications, Inc., New York, revised edition, 1976.
\newblock Presentations of groups in terms of generators and relations.

\bibitem{MattesW22}
Caroline Mattes and Armin Wei{\ss}.
\newblock Improved parallel algorithms for generalized {Baumslag} groups.
\newblock In {\em Proceedings of the 15th Latin American Symposium on
  Theoretical Informatics, {LATIN} 2022}, volume 13568 of {\em Lecture Notes in
  Computer Science}, pages 658--675. Springer, 2022.
\newblock \href {https://doi.org/10.1007/978-3-031-20624-5\_40}
  {\path{doi:10.1007/978-3-031-20624-5\_40}}.

\bibitem{MehlhornSU94}
Kurt Mehlhorn, R.~Sundar, and Christian Uhrig.
\newblock Maintaining dynamic sequences under equality-tests in polylogarithmic
  time.
\newblock In {\em Proceedings of the 5th Annual ACM-SIAM Symposium on Discrete
  Algorithms, SODA 1994}, pages 213--222. ACM/SIAM, 1994.

\bibitem{MehlhornSU97}
Kurt Mehlhorn, R.~Sundar, and Christian Uhrig.
\newblock Maintaining dynamic sequences under equality tests in polylogarithmic
  time.
\newblock {\em Algorithmica}, 17(2):183--198, 1997.
\newblock \href {https://doi.org/10.1007/BF02522825}
  {\path{doi:10.1007/BF02522825}}.

\bibitem{MyNiUs14}
Alexei~G. Myasnikov, Andrey Nikolaev, and Alexander Ushakov.
\newblock Knapsack problems in groups.
\newblock {\em Mathematics of Computation}, 84:987--1016, 2015.
\newblock \href {https://doi.org/10.1090/S0025-5718-2014-02880-9}
  {\path{doi:10.1090/S0025-5718-2014-02880-9}}.

\bibitem{MyUsWo11}
Alexei~G. Myasnikov, Alexander Ushakov, and Dong~Wook Won.
\newblock The word problem in the {Baumslag} group with a non-elementary {Dehn}
  function is polynomial time decidable.
\newblock {\em Journal of Algebra}, 345(1):324--342, 2011.
\newblock \href {https://doi.org/10.1016/j.jalgebra.2011.07.024}
  {\path{doi:10.1016/j.jalgebra.2011.07.024}}.

\bibitem{MyUsWo}
Alexei~G. Myasnikov, Alexander Ushakov, and Dong~Wook Won.
\newblock Power circuits, exponential algebra, and time complexity.
\newblock {\em International Journal of Algebra and Computation}, 22(6), 2012.
\newblock \href {https://doi.org/10.1142/S0218196712500476}
  {\path{doi:10.1142/S0218196712500476}}.

\bibitem{Olshanskii91}
Alexander~Yu. O{l\cprime}shanski\u{\i}.
\newblock Hyperbolicity of groups with subquadratic isoperimetric inequality.
\newblock {\em Internat. J. Algebra Comput.}, 1(3):281--289, 1991.
\newblock \href {https://doi.org/10.1142/S0218196791000183}
  {\path{doi:10.1142/S0218196791000183}}.

\bibitem{Olsh92}
Alexander~Yu. O{l\cprime}shanski\u{\i}.
\newblock Almost every group is hyperbolic.
\newblock {\em Internat. J. Algebra Comput.}, 2(1):1--17, 1992.
\newblock \href {https://doi.org/10.1142/S0218196792000025}
  {\path{doi:10.1142/S0218196792000025}}.

\bibitem{Papasoglu95}
Panagiotis Papasoglu.
\newblock On the sub-quadratic isoperimetric inequality.
\newblock In {\em Geometric group theory ({C}olumbus, {OH}, 1992)}, volume~3 of
  {\em Ohio State Univ. Math. Res. Inst. Publ.}, pages 149--157. de Gruyter,
  Berlin, 1995.
\newblock \href {https://doi.org/10.1515/9783110810820.149}
  {\path{doi:10.1515/9783110810820.149}}.

\bibitem{Pla94}
Wojciech Plandowski.
\newblock Testing equivalence of morphisms on context-free languages.
\newblock In {\em Proceedings of the 2nd Annual European Symposium on
  Algorithms, ESA 1994}, volume 855 of {\em Lecture Notes in Computer Science},
  pages 460--470. Springer, 1994.
\newblock \href {https://doi.org/10.1007/BFb0049431}
  {\path{doi:10.1007/BFb0049431}}.

\bibitem{Plat04}
A.~N. Platonov.
\newblock An isoparametric function of the {B}aumslag-{G}ersten group.
\newblock {\em Vestnik Moskov. Univ. Ser. I Mat. Mekh.}, 3:12--17, 70, 2004.

\bibitem{RiSe95}
Eliyahu Rips and Zlil Sela.
\newblock Canonical representatives and equations in hyperbolic groups.
\newblock {\em Inventiones Mathematicae}, 120:489--512, 1995.
\newblock \href {https://doi.org/10.1007/BF01241140}
  {\path{doi:10.1007/BF01241140}}.

\bibitem{Schl06}
Saul Schleimer.
\newblock Polynomial-time word problems.
\newblock {\em Commentarii Mathematici Helvetici}, 83(4):741--765, 2008.
\newblock \href {https://doi.org/10.4171/CMH/142} {\path{doi:10.4171/CMH/142}}.

\bibitem{Sims70}
Charles~C. Sims.
\newblock Computational methods in the study of permutation groups.
\newblock In {\em Computational {P}roblems in {A}bstract {A}lgebra ({P}roc.
  {C}onf., {O}xford, 1967)}, pages 169--183. Pergamon, Oxford, 1970.
\newblock \href {https://doi.org/10.1016/C2013-0-02156-1}
  {\path{doi:10.1016/C2013-0-02156-1}}.

\bibitem{Waa90}
Stephan Waack.
\newblock The parallel complexity of some constructions in combinatorial group
  theory.
\newblock {\em Journal of Information Processing and Cybernetics EIK},
  26:265--281, 1990.

\bibitem{WaeWei19}
Jan~Philipp W{\"{a}}chter and Armin Wei{\ss}.
\newblock An automaton group with {PSPACE}-complete word problem.
\newblock {\em Theory Comput. Syst.}, 67(1):178--218, 2023.
\newblock \href {https://doi.org/10.1007/S00224-021-10064-7}
  {\path{doi:10.1007/S00224-021-10064-7}}.

\bibitem{Weiss16}
Armin Wei{\ss}.
\newblock A logspace solution to the word and conjugacy problem of generalized
  {B}aumslag-{S}olitar groups.
\newblock In {\em Algebra and Computer Science}, volume 677 of {\em
  Contemporary Mathematics}, pages 185--212. American Mathematical Society,
  2016.
\newblock \href {https://doi.org/10.1090/conm/677}
  {\path{doi:10.1090/conm/677}}.

\bibitem{Wis09}
Daniel~T. Wise.
\newblock Research announcement: the structure of groups with a quasiconvex
  hierarchy.
\newblock {\em Electronic Research Announcements in Mathematical Sciences},
  16:44--55, 2009.
\newblock \href {https://doi.org/10.3934/era.2009.16.44}
  {\path{doi:10.3934/era.2009.16.44}}.

\end{thebibliography}

\end{document}